\documentclass[a4paper,11pt]{article}
\usepackage{amsmath,amsthm,amssymb,amsfonts}
\usepackage[svgnames,x11names,table]{xcolor}
\usepackage[utf8]{inputenc}
\usepackage{mathtools}
\usepackage{enumitem}
\usepackage{caption}
\usepackage{subcaption}

\usepackage{graphicx}
\usepackage{geometry}
\usepackage{framed}  
\usepackage[ruled,vlined]{algorithm2e}
\usepackage{booktabs}
\usepackage{cite}
\usepackage{hyperref} 
\hypersetup{
        colorlinks=true,   
        linkcolor=blue,    
        citecolor=green,    
        urlcolor=cyan      
      }
      
\usepackage{tikz}
\usetikzlibrary{positioning}

\usepackage{bbm}

\usepackage{xfrac}
\usepackage{microtype}

\renewcommand{\emptyset}{\varnothing}

\newtheorem{theorem}{Theorem}[section]
\newtheorem{lemma}[theorem]{Lemma}
\newtheorem{proposition}[theorem]{Proposition}
\newtheorem{corollary}[theorem]{Corollary}

\newtheorem{remark}[theorem]{Remark}

\newcommand{\R}{\mathbbm R}

\newcommand{\FBE}{\varphi_{\gamma}^{\mbox{\small{FB}}}}
\newcommand{\DRE}{\varphi_{\gamma}^{\mbox{\small{DR}}}}
\newcommand{\FBEk}{\varphi_{\gamma_k}^{\mbox{\small{FB}}}}
\newcommand{\DREk}{\varphi_{\gamma_k}^{\mbox{\small{DR}}}}

\newcommand{\FBEkk}{\varphi_{\mu_k}^{\mbox{\small{FB}}}}

\newcommand{\limout}{\limsup}
\newcommand{\liminn}{\liminf}

\newcommand{\tophi}{\xrightarrow[\varphi]{}}

\DeclareMathOperator{\prox}{prox}
\DeclareMathOperator{\proj}{proj}
\DeclareMathOperator{\dom}{dom}
\DeclareMathOperator{\epi}{epi}

\DeclareMathOperator{\dist}{dist}

\DeclareMathOperator*{\argmin}{arg\,min}

\newcommand{\x}{\xi}

\title{Understanding the Douglas-Rachford splitting method through the lenses of  Moreau-type envelopes}

\author{Felipe Atenas     \thanks{School of Mathematics \& Statistics,
                           The University of Melbourne,
                             Australia.
                            E-mail:~\href{href:felipe.atenas@unimelb.edu.au}
                                          {felipe.atenas@unimelb.edu.au} 
                                          }
                        \thanks{The author was supported in part by Australian Research Council grant DP230101749, and FAPESP grant 2019/20023-1}
}

\begin{document}
\maketitle

\begin{abstract}
We analyze the Douglas-Rachford splitting method for weakly convex optimization problems, by the token of the Douglas-Rachford envelope, a merit function akin to the Moreau envelope.  First, we use epi-convergence techniques to show that this artifact approximates the original objective function via epigraphs. Secondly, we present how global convergence and local linear convergence rates for Douglas-Rachford  splitting can be obtained using such envelope, under mild regularity assumptions. The keystone of the convergence analysis is the fact that the Douglas-Rachford envelope satisfies a sufficient descent inequality alongside the generated sequence, a feature that allows us to use arguments usually employed for descent methods. We report numerical experiments that use weakly convex penalty functions, which are comparable with the known behavior of the method in the convex case.
\end{abstract}


\section{Introduction and motivation}\label{sec1}

Decomposition techniques are fundamental in dealing with complex large-scale systems. Decomposition can be achieved by breaking up a problem in simpler  pieces, depending on the type of variables and constraints. Separability can also stem from structural properties, such as smoothness and convexity, or lack thereof.  Furthermore, splitting methods have been successfully applied in signal processing, image processing, and machine learning,  see \cite{combettes2007douglas,cai2010split,combettes2011proximal,boyd2011distributed,glowinski2017splitting} for a few illustrations of applications.

Operator splitting methods decompose complex structured problems into smaller
individual subproblems. A solution of the original problem is obtained by iteratively
solving separate subproblems for each involved function or operator. Prominent instances suitable for composite optimization are the Douglas-Rachford (DR), the forward-backward (FB), and the Peaceman-Rachford splitting methods, the Alternating Direction Method of
Multipliers (ADMM), and the Spingarn's partial inverse method; we refer to
\cite{lions1979splitting,spingarn1983partial,eckstein1989splitting,boyd2011distributed,davis2016convergence}  
and references therein for  further details.
When applied to optimization problems, these methods were  originally studied for linear, and
more generally, convex programming.

The cornerstone of most operator splitting methods is the proximal point algorithm (PPA), introduced in \cite{martinet-1970} and thoroughly
studied in \cite{rockafellar1976monotone} to find
a zero of a maximal monotone operator. In the context of convex optimization, for proper lower semicontinuous (lsc) convex 
functions $f, g: \R^d \to \R \cup \{+\infty\}$, the problem boils down to
minimizing the composite function $f + g$. As explained in 
\cite{eckstein1989splitting}, one DR splitting iteration (given in the scheme
\eqref{DR:scheme} below) amounts to applying the PPA with constant stepsize equal to $1$, to an auxiliary maximal monotone operator. In this manner, the DR splitting method provably inherits properties of the PPA in the convex case, such as convergence and  convergence rates.

We focus on problems involving the sum of two terms, the more classical setting. More recently, extensions to the sum of more than two functions/monotone operators have been proposed, see
\cite{eckstein2017simplified,malitsky2022resolvent} for some examples. In this way, we consider the following minimization problem \begin{equation}
	\label{problem:primal-DRS} \min_{x \in \R^d} \varphi(x) = f(x) +
	g(x), \end{equation}
where $f,g: \R^d \to \R \cup \{+\infty\}$ are proper lsc functions, not necessarily convex.

Given a relaxation parameter $\lambda >0$, a stepsize $\gamma
>0,$ and an initial point $z^0 \in \R^d$, we define one iteration of the (relaxed) DR splitting method  \cite{themelis2020douglas} as follows: 
\begin{equation} \label{DR:scheme}  \left\{\begin{aligned} x^k  &\in  \prox_{\gamma
			f}(z^k)& \\ y^k  &\in  \mbox{prox}_{\gamma g}(2x^k - z^k)& \\
		z^{k+1} & =  z^k + \lambda(y^k - x^k).& \end{aligned}\right.  \end{equation} Note that for $\lambda =1$, scheme \eqref{DR:scheme} reduces to plain DR splitting.  As stated,  one iteration amounts to performing  two  successive proximal steps, 
computed separately for each term of the sum, followed by a fixed point/correction step. In the district of proximal splitting methods, the DR method has a close neighbor,  the FB method. This method defines the iterates as $x^{k+1} = \prox_{\gamma g}\big(x^k - \gamma \nabla f(x^k)\big)$, and it is also known as the proximal-gradient method.

The iterative approach \eqref{DR:scheme} has a long history.  Lions and Mercier in \cite{lions1979splitting}  studied convergence properties and speed of convergence for splitting methods to find a zero of the sum of two maximally monotone operators
defined on a Hilbert space. As a special case, this analysis covers  optimization problems of the form
\eqref{problem:primal-DRS} for proper lsc convex functions, where the corresponding operators are the subdifferentials of convex analysis of the involved functions. 
Under mild regularity assumptions, the DR sequence $\{z^k\}$
converges to some $z^{\star}$, for which  $x^{\star} = \prox_{\gamma
	f}(z^{\star})$ solves \eqref{problem:primal-DRS}, and both
$\{x^{k}\}$ and $\{y^{k}\}$ converge to $x^{\star}$ \cite[Theorem
3.15, Proposition 3.40]{eckstein1989splitting}.  Additionally, if $f$ is
differentiable and strongly convex, with  Lipschitz
continuous gradient, then $\{z^{k}\}$ $Q-$linearly converges to $z^{\star}$, and
$\{x^k\}$ $Q-$linearly converges to the unique solution to
\eqref{problem:primal-DRS}. For varying stepsizes  and inexact proximal
evaluations, see \cite[Theorem 7]{eckstein1992douglas}.    As for operators, the authors in \cite{dao2019adaptive} analyze the convergence of the DR method to find a zero of the sum of a strongly monotone and a weakly monotone operator, providing linear rates of convergence.

Typically, the analysis of DR splitting relies on the study of the sequence of distances from the iterates to the solution set, and its monotonicity. However, the sequence of
function values is not monotonically nonincreasing necessarily, and for this reason 
\cite{patrinos2014douglas} introduced
a special merit function, called \textit{Douglas-Rachford envelope} (DRE), that shows a monotone behavior alongside iterations.  For convex composite problems with a convex quadratic term, 
the DRE is real-valued and continuously differentiable. Furthermore,
one DR splitting iteration corresponds to one variable metric gradient step applied to the DRE \cite[equation (17)]{patrinos2014douglas}.
In a manner similar to how the Moreau envelope sheds a light on the PPA,
the DRE gives an insight on DR. In particular, because the DR splitting method provides (sufficient) descent for the  DRE  merit function, a variable metric 
gradient method for the DRE yields 
complexity estimates and rates of convergence for the DR iterates. A point
crucial for this type of  analysis is that DRE critical points are
related to minimizers of the original convex problem. 
For convex composite objective functions with a (sufficiently) smooth and strongly convex term, a similar approach is adopted in \cite{patrinos2014forward} to 
analyze the FB splitting method by means of a suitably defined envelope (see \eqref{FBE}),  called \emph{forward-backward envelope (FBE)}.

The literature is much more scarce
for nonconvex problems, the setting considered in this work. The authors in
\cite{li2016douglas,li2015global} study the sum of a differentiable function with Lipschitz continuous gradient, 
and a proper lsc function with an easily computable proximal operator. By defining a merit function related to the DRE,  subsequential convergence to a critical point is obtained, as well as eventual convergence rates under some extra assumptions, namely, that the functions satisfy the K{\L} inequality \cite{kurdyka1998gradients,bolte2007clarke}, a concept related to error bounds \cite[Theorem 4.1]{li2018calculus}. These two notions are often used in the literature to establish local rates of convergence \cite{atenas2023unified,luo1993error,robinson1999linear,frankel2015splitting,davis2015convergence,bolte2017error,li2018calculus}.

The purpose of this work is to provide a deeper understanding of splitting methods by using Moreau-type envelopes, in particular, our focus being the DR splitting method. We first start examining how the epigraph of the FBE and DRE approximate the epigraph of the composite objective function, shedding a light on why these envelopes are quite successful not only in explaining the convergence of the splitting methods in nonconvex settings, but also in defining  appropriate approximations of the objective function. This is a consequence of the nature of the envelopes themselves, since they can be interpreted as extensions of the Moreau envelope (defined in \eqref{def:Moreau}) tailored to composite optimization problems. We follow the approach developed in \cite{burke2013epi}.

The second part of our contributions, independent of the first part, refers to deriving local convergence rates of the DR splitting method for weakly convex optimization using envelopes. This is achieved by combining the machinery developed in \cite{themelis2020douglas} for the DRE with 
the unifying framework for descent methods from \cite{atenas2023unified}. Our results resemble the ones briefly referred without proof in \cite[page 15]{themelis2020douglas}
for semialgebraic functions, and the ones presented in \cite{li2016douglas} using a different Lyapunov-type merit function.  To the best of the author's knowledge, this is the first result of local rates of convergence for function values (and iterates) of the DR method for nonconvex problems.

The remainder of this work is organized as follows. We introduce the notation and 
background on variational analysis in Section~\ref{section:background}. Next, in Section~\ref{s:epi-approx} we investigate the envelopes for the FB and the DR methods from a theoretical point of view. In particular, we establish that both FBE and DRE epi-approximate the objective function of interest, and explore some consequences of this fact. In Section~\ref{s:conv-DR} we show the necessary components to follow the ideas of \cite{atenas2023unified} to obtain convergence and local rates of convergence for the nonconvex DR splitting method, under the assumption of a mild regularity condition. We then proceed in Section~\ref{section:numerical} to show the numerical performance of the DR method when applied to a regularized linear regression problem using weakly convex penalties.  We close in Section~\ref{section-final} with some concluding remarks and possible extensions.


\section{Background material} \label{section:background}

\subsection{Notation and basic definitions}

Unless stated otherwise, $\langle \cdot, \cdot \rangle$ denotes an inner product  in $\R^d$, and $\| \cdot \|$ its induced norm.  A function $\varphi: \R^d \to \R \cup \{+\infty\}$ is called proper whenever
a point $x \in \R^n$ exists such that $\varphi(x) < +\infty.$ The set of such points satisfying this condition, the domain of $\varphi$, is denoted by $\dom(\varphi)$. The epigrapgh of $\varphi$ is defined as $\epi(\varphi) = \{ (x,\alpha) \in \R^d \times \R: \varphi(x) \le \alpha \}$. 

For a set $C \subseteq \R^d$, and a sequence of sets $\{C^k \subseteq \R^d\}$, the inner limit of $\{ C^k\}$, denoted $\displaystyle\liminf_{k\to +\infty} C^k$, is the set of all limit points of sequences $\{x^k \in C^k\}$, that is, \[ \liminf_{k\to+\infty} C^k = \{x \in \R^d | \: \exists \{x^k\} \mbox{ such that } C^k \ni x^k \to x \mbox{ as } k \to +\infty\}.\] The outer limit of $\{ C^k\}$, denoted $\displaystyle\limsup_{k\to+\infty} C^k$, is the set of all accumulation points of sequences $\{x^k\}_k$ with elements in $C^k$ throughout the respective convergent subsequence, that is, \[ \limsup_{k\to+\infty} C^k = \{x \in \R^d | \: \exists \{x^k\}  \mbox{ and a subsequence } C^{k_j} \ni x^{k_j} \to x \mbox{ as } j \to +\infty\}.\]  Note that $\displaystyle\liminf_{k\to+\infty} C^k \subseteq \displaystyle\limsup_{k\to+\infty} C^k$. We say $\{C^k\}$ set-converges to $C$ if $\displaystyle\limout_{k\to+\infty} C^k \subseteq C \subseteq \displaystyle\liminf_{k\to+\infty} C^k$. A sequence of functions $\{\varphi^k: \R^d \to \R \cup \{+\infty\}\}$ is said to epi-approximate $\varphi$ (cf. \cite[eq. (3.1)]{burke2013epi}) if $\{\varphi^k\}$ epi-converges to $\varphi$, that is,  $\{\epi(\varphi^k)\}$ set-converges to $\epi(\varphi)$. As an example, any proper lsc convex function can be epi-approximated by its Moreau envelope, defined below in \eqref{def:Moreau} (see \cite[Corollary 4.8]{burke2013epi}). In Section~\ref{s:epi-approx}, we prove that both the FBE and the DRE epi-approximate  $ f + g$ in a broader setting.

For a proper function $\varphi: \R^n \to \R \cup \{+\infty\}$, and $x \in \dom(\varphi)$, we denote by  $\hat{\partial} \varphi$ the Fr\'echet (or regular)
subdifferential of $\varphi$, defined as \begin{equation*} \hat{\partial}
	\varphi(x) = \left\{ v \in \R^d : \liminf_{y \to x}
	\frac{\varphi(y)-\varphi(x) - \langle v, y - x \rangle}{\|
		y - x\|} \ge 0\right\}.  \end{equation*}   The limiting (or general) subdifferential of $\varphi$, denoted $\partial \varphi$, is defined as     \begin{equation*} \partial \varphi(x) = \limsup_{y \tophi x} \hat{\partial}\varphi(x),
\end{equation*} where $y \tophi x$ denotes convergence in the attentive sense, that is, $y \to
x $ and $\varphi(y) \to \varphi(x)$. If $\varphi$ is proper lsc convex, then the Fr\'echet and limiting subdifferentials coincide with the subdifferential of convex analysis, namely, the set of vectors $v \in \R^n$ such that for all $y \in \R^n$, $\varphi(y) \ge \varphi(x) + \langle v, y - x \rangle$. When $\varphi$ is differentiable, then $\hat{\partial}\varphi(x) = {\partial}\varphi(x) = \{\nabla f (x)\}$. A notion that is crucial for calculus rules for subdifferentials is the horizon subdifferential. We denote by $\partial^{\infty} \varphi( x)$ the set of horizon subgradients $v$ of $\varphi$ at $ x \in \dom(\varphi)$, vectors for which there exist sequences $x^k \tophi x$, and $t_k \downarrow 0$, such that $v^k \in \partial \varphi (x^k)$, and $t_k v^k \to v$.

We say a point $x$ is critical for $\varphi$ if $0 \in \partial \varphi (x).$ This condition is necessary for a point to be a local minimum. For a critical point $x$ of $\varphi$, the real number $\varphi(x)$ is called critical value.

A special family of differentiable functions is the continuously differentiable functions with Lipschitz continuous gradient. We say a function $f: \R^d \to \R$ is $L_f$-smooth for $L_f>0$, if $f$ is  continuously differentiable, and for all $x,y \in \R^d$, $\|\nabla f(x) - \nabla f(y)\| \le L_f \| x-y\|$. In the next lemma, we state a powerful tool supplied by $L$-smooth functions \cite[Lemma 2.64]{bauschke2011convex}.

\begin{lemma}[Descent lemma] \label{descent-lemma}
	Let $f: \R^d \to \R \cup \{+\infty\}$ be $L_f$-smooth function. Then, for all $x,y \in \R^d$, \begin{equation} 
		|f(y)-f(x)-\langle \nabla f(x), y-x \rangle | \le \dfrac{L_f}{2}\|x-y\|^2.
	\end{equation}
\end{lemma}

For the following family of  functions with a \textit{benign} form of nonconvexity, we can exploit the convex analysis machinery. A function $\varphi: \R^d \to \R\cup\{+\infty\}$ is said to be $\rho-$weakly convex, for $\rho >0$, if $\varphi(\cdot) + \frac{\rho}{2}\|\cdot\|^2$ is convex. Weakly convex functions can be characterized in a variety of forms \cite{davis2019stochastic}. In particular, a function $\varphi$ is $\rho-$weakly convex, if and only if, for any $y,x \in \mbox{dom}(\varphi),$ and $v \in \partial \varphi(x),$\begin{equation} \label{prox-subgrad}
	\varphi(y) + \frac{\rho}{2}\| y-x \|^2 \ge \varphi(x) + \langle v, y- x\rangle.
\end{equation}
Examples of weakly convex functions are convex functions ($\rho=0$) and $L$-smooth functions ($\rho=L$) \cite[Proposition 2.4]{atenas2023unified}. More properties and examples of weakly convex functions can be found in \cite{hoheiselproximal,davis2019stochastic}. See also Section~\ref{section:numerical} for two examples of weakly convex functions used  as regularizers in regression problems.

\subsection{Proximal operator and Moreau envelope}

As stated in \eqref{DR:scheme}, one iteration of DR splitting is constructed using the
proximal operator, defined as follows. For a point $x \in \R^d$, $\prox_{\gamma \varphi}(x)$ denotes the proximal point operator of $\varphi$ for a stepsize parameter $\gamma >0$ evaluated at $x$, and defined by \begin{equation} \label{def:prox}
	\prox_{\gamma \varphi}(x) = \argmin_{y \in \R^n} \left\{ \varphi(y) + \frac{1}{2\gamma}\| y - x\|^2\right\}.
\end{equation}
The optimal value function of the minimization problem in \eqref{def:prox} defines the Moreau envelope of $\varphi$ with stepsize $\gamma >0$,  denoted by $e_{\gamma}\varphi$. More precisely, 
\begin{equation} \label{def:Moreau}e_{\gamma}\varphi(x) = \inf_{y \in \R^d} \left\{ \varphi(y) + \frac{1}{2\gamma}\|y - x\|^2\right\}. \end{equation} In the general case, the operator $\prox_{\gamma \varphi}: \R^n \rightrightarrows \R^n$ could be empty-valued and $e_{\gamma}\varphi$ might take the value $-\infty$. The proximal operator and the Moreau envelope are well-defined for prox-bounded functions. We say $\varphi$ is prox-bounded if $\varphi(\cdot) + \frac{1}{2\gamma}\|\cdot\|^2$ is bounded from below for some $\gamma >0$, and the supremum $\gamma_{\varphi}>0$ of such parameters is called the threshold of prox-boundedness. In this manner, if $\varphi$ is a proper lsc prox-bounded function with threshold $\gamma_{\varphi} > 0$, then for any $\gamma \in (0, \gamma_{\varphi})$,  the images of $\prox_{\gamma \varphi}$ are nonempty and compact, and $e_{\gamma}\varphi$ is finite-valued \cite[Theorem 1.25]{rockafellar2009variational}. A particular case is when $\varphi$ is a proper lsc convex function, in which case $\gamma_{\varphi}=+\infty$,  and $\prox_{\gamma \varphi}$ is a single-valued mapping \cite[Theorem 12.12, Theorem 12.17]{rockafellar2009variational}. As a consequence, when $f: \R^d \to \R$ is a $L_f$-smooth function, and $g: \R^d \to \R \cup \{+\infty\}$ is a $\rho$-weakly convex function, both proximal steps in \eqref{DR:scheme} are uniquely defined for any $\gamma \in ( 0, \min\{L_f^{-1}, \rho^{-1}\} )$.

The proximal operator enjoys a plethora of calculus rules. The following result is standard for convex functions, and since it can be easily proven using the optimality conditions of the proximal operator, it also holds in a broader setting. We resort to this lemma to prove the next result.

\begin{lemma} \label{prox-linear-translation}
	Let $G: \R^d \to \R \cup \{ +\infty\}$ be a  proper lsc prox-bounded function with threshold $\gamma_G >0$. Define $G_a(\cdot) = G(\cdot) + \langle a, \cdot \rangle $ for some $a \in \R^d$. Then, for all $\gamma \in (0, \gamma_G)$, \begin{equation*}
		\prox_{\gamma G_a}(\cdot) = \prox_{\gamma G}(\cdot - \gamma a).
	\end{equation*}
\end{lemma}

Likewise, the next result is valid in the convex case \cite[Theorem 23.47]{bauschke2011convex}, and we extend it here for weakly convex functions. 

\begin{lemma} \label{prox-convergence}
	Let $G: \R^d \to \R \cup \{ +\infty\}$ be a  proper lsc $\rho-$weakly convex function, $\{x^k\}$ a sequence and $\bar x \in \R^d$, such that $x^k \to \bar x$. Given a sequence of real numbers $\{\gamma_k\}$ such that $\gamma_k \downarrow 0$, then \begin{equation*}
		\prox_{\gamma_k G}(x^k) \to \proj_{\overline{\dom(\partial G)}}(\bar x).
	\end{equation*} 
\end{lemma}

\begin{proof}
	Denote $\mu_k^{-1} = \gamma_k^{-1} - \rho$, in such a way that for all $k$ sufficiently large, $\mu_k > 0$ and $\mu_k \downarrow 0$. Then \begin{equation*}
		\begin{array}{rcl}
			\prox_{\gamma_k G}(x^k) & = & \displaystyle\argmin_{x \in \R^d} \left\{ G(x) + \dfrac{\rho}{2}\| x - x^k\|^2 + \dfrac{1}{2\mu_k}\|x - x^k\|^2 \right\} \\
			& = & \prox_{\mu_k(G_{\rho}(\cdot) - \rho 
				\langle x^k , \cdot \rangle)}(x^k)\\
			& = & \prox_{\mu_kG_{\rho}(\cdot) }(x^k+ \rho\mu_k x^k),
		\end{array}
	\end{equation*} where $G_{\rho}(\cdot) = G(\cdot) + \frac{\rho}{2}\| \cdot \|^2$, and the last line follows from Lemma~\ref{prox-linear-translation}. Since $G_{\rho}$ is convex, and $x^k + \rho \mu_k x^k \to \bar{x}$, the conclusion follows from \cite[Theorem 23.47]{bauschke2011convex} and $\dom(\partial G) = \dom(\partial G_{\rho})$.
\end{proof}


\section{Epigraphic approximation through envelopes} \label{s:epi-approx}

The approximation theory for optimization problems via set-convergence has been explored in \cite{royset2020set}. In this section, we first introduce the aforesaid envelopes for the FB and DR methods, and state some properties. Then, we prove that the FBE epi-approximates the original function and, as a consequence, that the DRE enjoys the same property. The relevance of this result is that these envelopes not only coincide with the objective function at a global minimum (see \eqref{DRE:properties}), but also their epigraphs asymptotically coincide with the epigraph of the objective function, a much stronger notion of approximation as discussed in \cite{royset2020set}. The foundation of this behavior is the fact that the Moreau envelope itself shows it. In some sense, it should be expected that \emph{any} method of proximal type satisfies the epi-approximation property.

\subsection{Moreau-type envelopes for splitting methods}

As stated in the introduction, the FB splitting method has $x \mapsto \prox_{\gamma g}(x-\gamma\nabla f(x))$ as iteration operator. A merit function resembling the Moreau envelope can be defined for the FB method, corresponding to a value function associated with the itetarion operator. More precisely, given  $x\in \R^d$ and $\gamma>0$, the FBE \cite{patrinos2013proximal,themelis2018forward}  is defined as  \begin{equation} \label{FBE} \FBE(x) =
	\min_{y \in \R^d}\left\{ f(x) + \langle  \nabla f(x), y - x \rangle +
	g(y) + \dfrac{1}{2\gamma}\|y-x\|^2\right\}.  \end{equation} Problem \eqref{FBE} can be seen as an approximation of 
$e_{\gamma\varphi}(x)$: $f$ is replaced by its first-order Taylor model
of at $x$. Observe that any minimizer of the problem in \eqref{FBE} defines the next step of the FB splitting method. As an optimal value function, the FBE is a real-valued locally Lipschitz penalty  \cite[Proposition 4.2]{themelis2018forward} when $f: \R^d \to \R$ is $L_f$-smooth, $g: \R^d \to \R \cup \{+\infty\}$ is proper lsc prox-bounded with parameter $\gamma_g >0$, and $\gamma \in (0, \min\{L_f^{-1},\gamma_g\} )$. Similarly, for any $z\in \R^n$ and $\gamma > 0$, the DRE \cite{patrinos2014douglas,themelis2020douglas} is defined, for $x = \prox_{\gamma f}(z)$, as  \begin{equation} \label{DRE} \DRE(z) =
	\min_{y\in\R^d}\left\{ f(x) + \langle  \nabla f(x), y - x \rangle +
	g(y) + \dfrac{1}{2\gamma}\|y-
	x\|^2\right\}.  \end{equation}  Problem \eqref{DRE} can be interpreted as an approximation of 
$e_{\gamma\varphi}(x)$ as well. Namely, in the sum $\varphi =
f + g$, $f$ is replaced by its first-order Taylor model
of at $x = \prox_{\gamma f}(z)$. Under similar assumptions for the FB method, the DRE is also finite-valued and locally Lipschitz continuous \cite[Proposition 3.2]{themelis2020douglas}. Furthermore, it is straightforward that the following relationship between envelopes holds \begin{equation} \label{DRE-FBE}
	\DRE(z) = \FBE(\prox_{\gamma f}(z)).
\end{equation} An advantage of the DRE (and the FBE) over other merit functions defined for splitting methods, is that it corresponds to a natural extension of the Moreau envelope in the following sense. Since $f$ is $L_f-$smooth,
$\prox_{\gamma f}$ is a single-valued operator \cite[Proposition
2.3(i)]{themelis2020douglas}, and from the update rule of $x^k$ in
\eqref{DR:scheme},  \begin{equation} \label{u-OC} x^k = \prox_{\gamma f}(z^k) \iff 0 =
	\gamma\nabla f(x^k) + x^k - z^k.  \end{equation} Combining this last identity with the update rule of $y^k$ in
\eqref{DR:scheme}, then $y^k$ is a solution to the following problem  \begin{equation} \label{DRE:moreau} \min_{y \in \R^d} \left\{ g(y) +
	\frac{1}{2\gamma}\|y - \big(x^k - \gamma \nabla f(x^k)\big)\|^2
	\right\}.  \end{equation} The optimal value of this problem is exactly $e_{\gamma g}(x^k - \gamma \nabla f(x^k))$, and the same holds for the FBE.  After expanding squares in the above expression, we end up with the form
of the DRE presented above. 

Another interpretation of the DRE is available when rewriting it as a penalty function in explicit form \cite{themelis2020douglas}. For this purpose, we first reformulate  problem \eqref{problem:primal-DRS} as
follows \begin{equation*} \min_{x,y \in \R^d} f(x) + g(y) \quad
	\mbox{ s.t. } \quad x - y =0.  \end{equation*} The augmented Lagrangian of this reformulation is, for $\beta
>0:$ \begin{equation*} \mathcal{L}_{\beta}(x,y,w) = f(x) + g(y)
	+ \langle w, x-y \rangle + \dfrac{\beta}{2}\| x - y\|^2, \end{equation*}  where $w \in \R^d$ is a Lagrange multiplier associated with the constraint $x-y
= 0$. Therefore, by the definition of the DRE, it holds for all $k$: \begin{equation}
	\label{DR:Lagrangian} \DRE(z^k) = \mathcal{L}_{\gamma^{-1}}\bigl(x^k, y^k,
	\gamma^{-1}(x^k-z^k)\bigr).  \end{equation}

Regarding convergence, one of the crucial properties of the DRE (and FBE) is that it preserves the minimizers of the original objective function, whenever they exist \cite[Theorem 3.4]{themelis2020douglas}, \cite[Theorem 4.4]{themelis2018forward}. More precisely, when $f$ is $L_f$-smooth, $g$ is proper lsc prox-bounded with threshold $\gamma_g$, and $\argmin \varphi \neq \emptyset$, then for any $\gamma \in (0, \min\{L_f^{-1},\gamma_g\})$, \begin{equation} \label{DRE:properties}
	\displaystyle\inf \varphi = \inf\FBE = \inf \DRE ,\quad  \argmin  \varphi =
	\argmin \FBE  =
	\prox_{\gamma f}(\argmin \DRE). 
\end{equation} In particular, whenever $\varphi$ is bounded below, so are
the FBE and the DRE. Although crucial to obtain convergence to critical points of the original problem via envelopes, the above identities only regard the behavior \emph{at} minimizers of $\varphi$ and the envelopes.  In the following section, we prove that these envelopes actually approximate the objective function via epigraphs, a more robust property.

\subsection{Epigraphic approximation via FBE and DRE}

We start the analysis by proving that the FBE of $\varphi = f +g $ epi-approximates $\varphi$ in a prox-friendly nonconvex convex setting, and then extend the results for the DRE. First, we prove a technical result from which can be deduced that the FBE is an \emph{outer} epi-approximation of the objective function.

\begin{lemma} \label{lemma:epi-inclusion}
	
	Suppose $f: \R^d \to \R $ is a $L_f$-smooth function, $g: \R^d \to \R \cup \{+\infty\}$ is a proper lsc prox-bounded function, and $\varphi = f + g$. Then,\begin{equation*}
		\epi(\varphi) + \liminn_{\gamma \downarrow 0}\epi\left(\dfrac{1}{2\gamma}\| \cdot \|^2\right) \subseteq \liminn_{\gamma \downarrow 0}\epi(\FBE).
	\end{equation*}
\end{lemma}

\begin{proof}
	Let $(x,\xi) \in \epi(\varphi)$ and  $(y, \eta) \in \liminn_{\gamma \downarrow 0}\epi(\|\cdot\|^2/(2\gamma))$. Take any $\gamma_k\downarrow 0$ and $(y^k, \eta^k) \to (y, \eta)$, such that $(y^k, \eta^k) \in \epi(\|\cdot\|^2/(2\gamma_k))$ for all $k$. Consequently, \[ \varphi(x) + \frac{1}{2\gamma_k}\|y^k\|^2 \le \xi + \eta^k.\]  In view of Lemma~\ref{descent-lemma}, it holds for all $z \in \R^d$,\[ f(x+y^k)  + \langle \nabla f(x + y^k), z -(x+y^k)\rangle + g(z) \le \varphi(z) + \frac{L}{2}\| z- (x+y^k)\|^2. \] Define $\{\mu_k\}$ such that  $\gamma_k^{-1} = \mu_k^{-1} + L$ for all $k$, that is, $\gamma_k = \frac{\mu_k}{1+ \mu_k L}$.  Note that for all sufficiently large $k$, $\mu_k > 0$, and thus $\mu_k \downarrow 0$. Combine the latter estimate with  the definition of the FBE to obtain, for all sufficiently large $k$, \[ \begin{array}{rcl}
		\FBEkk(x+y^k) &\le& \displaystyle\inf_{z  \in \R^d} \left\{ \varphi(z) + \frac{1}{2}\left(\frac{1}{\mu_k} + L \right)\|z - (x+y^k) \|^2 \right\} \\
		&=& \displaystyle\inf_{z  \in \R^d} \left\{ \varphi(z) + \frac{1}{2\gamma_k} \|z - (x+y^k) \|^2 \right\} \\
		& \le & \varphi(x) + \dfrac{1}{2\gamma_k} \|y^k \|^2 \\
		& \le & \xi + \eta^k.
	\end{array}\] Hence, $(x+y^k, \xi+ \eta^k) \in \epi(\FBEkk)$, and $(x+y^k, \xi+ \eta^k) \to (x+y, \xi + \eta)$ as $\mu_k \downarrow 0$, or equivalently, $(x+y, \xi + \eta) \in \liminn_{\mu \downarrow 0} \epi\big(\FBE\big)$.
\end{proof}

Next, we prove that the FBE is an \emph{inner} epi-approximation of $\varphi$ for weakly convex problems, and thus it defines an epi-approximation (cf. \cite[Proposition 4.5]{burke2013epi}).

\begin{theorem}[Epi-approximation: forward-backward envelope] \label{FBE-epi-convergence}
	
	Suppose $f: \R^d \to \R $ is a $L_f$-smooth function, $g: \R^d \to \R \cup \{+\infty\}$ is a proper lsc $\rho$-weakly convex function, and $\varphi = f + g$. Then, for any $\gamma_k \downarrow 0$, $\{\FBEk\}$ epi-approximates $\varphi$.
\end{theorem}

\begin{proof}
	First,  it holds that \[\liminn_{\gamma \downarrow 0} \epi(\FBE) \supseteq \epi(\varphi).\] Indeed, since $\|0\|^2/(2\gamma) \le 0$, then $(0,0) \in \epi\left(\dfrac{1}{2\gamma}\| \cdot \|^2\right)$, and thus Lemma~\ref{lemma:epi-inclusion} gives the inclusion. Next, we prove \[\limout_{\gamma \downarrow 0} \epi(\FBE) \subseteq \epi(\varphi).\]  Let  $(\bar{x}, \bar{\alpha}) \in \limout_{\gamma \downarrow 0} \epi(\FBE)$. Then, up to a subsequence, there exist $\gamma_k \downarrow 0$, $x^k \to \bar{x}$, and $\alpha_k \to \bar{\alpha}$, such that $\FBEk(x^k) \le \alpha_k$. Setting $z^k = \prox_{\gamma_k g}(x^k - \gamma_k \nabla f(x^k ))$, note that $z^k \to \proj_{\overline{\dom(\partial g)}}(\bar{x})$ in view of Lemma~\ref{prox-convergence} and $x^k - \gamma_k \nabla f(x^k ) \to \bar{x}$. Observe that, actually,  $z^k \to \bar{x}$. Indeed,  by way of contradiction,  assume that $\{z^k - x^k\}$ stays bounded away from $0$. By definition of $z^k$,  \[ f(x^k) + \langle \nabla f(x^k),  z^k - x^k \rangle + g(z^k) +  \frac{1}{2\gamma_k}\|z^k - x^k\|^2 \le \alpha_k,\] and Lemma~\ref{descent-lemma}, \begin{equation*}  
		f(x^k) + \langle \nabla f(x^k),  z^k - x^k \rangle \ge f(z^k) - \dfrac{L}{2} \| z^k - x^k \|^2,
	\end{equation*} it follows that  \begin{equation} \label{epi-aux-2}  \varphi(z^k) + \frac{1}{2}\left(\frac{1}{\gamma_k}-L\right)\|x^k - z^k\|^2 \le \alpha_k.\end{equation} Since $g$ is $\rho-$weakly convex, there exists an affine function $\ell_{\varphi}$ that globally lower bounds $\varphi + \frac{\rho}{2}\|\cdot\|^2$. In this way,  \eqref{epi-aux-2} implies \begin{equation*}
		\ell_{\varphi}(z^k) + \dfrac{1}{2}\left( \dfrac{1}{\gamma_k} - L \right)\| z^k - x^k \|^2\le \alpha_k + \dfrac{\rho}{2}\|z^k\|^2.
	\end{equation*} Take $\mu_k \in \R$ such that $\mu_k^{-1} = \gamma_k^{-1} - L$. For all sufficiently large $k$, $\mu_k > 0$, and also $\mu_k \downarrow 0$. Rearranging terms in the estimate above and dividing by $\| z^k - x^k\|$, it yields \begin{equation*}
		\dfrac{1}{2\mu_k}\| z^k - x^k \|   \le \dfrac{\alpha_k-\ell_{\varphi}(z^k) }{\| z^k - x^k\|}  + \dfrac{\rho}{2}\dfrac{\|z^k\|^2}{\| z^k - x^k\|}.
	\end{equation*} Since both $\{z^k\}$ and $\{\alpha_k\}$ are convergent sequences, the assumption on $\{z^k-x^k\}$ implies that the right-hand side of the last estimate is bounded, while the left-hand side is unbounded, yielding a contradiction. Therefore, $z^k - x^k \to 0$, and thus $z^k \to \bar x$. Furthermore, \eqref{epi-aux-2} implies that $\{\varphi(z^k)\}$ is upper bounded, and since $\varphi$ is lsc, then $\varphi(\bar{x}) \le \displaystyle\liminf_{k\to +\infty} \varphi(z^k) < +\infty$. Hence $\bar{x} \in \dom(\varphi)$. In order to show that $(\bar{x}, \bar{\alpha}) \in \epi(\varphi)$,   note that \eqref{epi-aux-2} is equivalent to  \begin{equation} \label{epi-aux-3}\dfrac{1}{2\mu_k}\|x^k - z^k\|^2 \le \alpha_k - \varphi(z^k).\end{equation}  Up to a subsequence if necessary, $\varphi(z^k) \to \displaystyle\liminf_{k \to +\infty} \varphi(z^k)$. Then, from \eqref{epi-aux-3}, 
	it holds that $0  \le \bar{\alpha} - \displaystyle\liminf_{k \to +\infty} \varphi(z^k) \le \bar{\alpha} - \varphi(\bar x)$, and thus $(\bar x, \bar \alpha) \in \epi(\varphi)$.

	In this manner, \[ \limout_{\gamma \downarrow 0} \epi(\FBE) \subseteq \epi(\varphi) \subseteq \liminn_{\gamma \downarrow 0} \epi(\FBE),\] and the conclusion follows.
\end{proof}

In view of the intrinsic relationship between the forward-backward and the Douglas-Rachford splitting methods, the DRE enjoys the same type of epi-approximation properties of the FBE, as stated in next result.

\begin{corollary}[Epi-approximation: Douglas-Rachford envelope] \label{DRE-epi-convergence}
	Suppose $f: \R^d \to \R $ is a $L_f$-smooth function, $g: \R^d \to \R \cup \{+\infty\}$ is a proper lsc $\rho$-weakly convex function, and $\varphi = f + g$. Then, for any $\gamma_k \downarrow 0$, $\{\DREk\}$ epi-approximates $\varphi$.
\end{corollary}

\begin{proof}
	Let $(z,\alpha) \in \limout_{\gamma \downarrow 0}\epi\big(\DRE\big)$,  then there exist $\gamma_k \downarrow 0$ and, up to a subsequence, $(z^k, \alpha^k) \in \epi\big(\DREk\big)$,  $z^k \to z$, and $\alpha^k \to \alpha$.
	In view of \eqref{DRE-FBE}, \begin{equation*}
		\epi\big(\DRE\big) =  \epi\big(\FBE\circ \prox_{\gamma f}\big),
	\end{equation*} and thus $(\prox_{\gamma_k f}(z^k), \alpha^k) \in \epi\big(\FBEk\big).$ Taking the limit, Theorem~\ref{FBE-epi-convergence} and Lemma~\ref{prox-convergence} imply \begin{equation*}
		(z, \alpha) \in \limout_{\gamma \downarrow 0} \epi\big(\FBEk\big) = \epi(\varphi).
	\end{equation*} On the other hand, take $(z,\alpha) \in \epi(\varphi)$. In view of Theorem~\ref{FBE-epi-convergence}, there exist $\gamma_k \downarrow 0$,  $z^k \to z$ and $\alpha^k \to \alpha$, such that $(z^k,\alpha^k)\in\epi\big(\FBEk\big)$. Set $w^k = z^k + \gamma_k \nabla f(z^k)$, then $z^k = \prox_{\gamma_k f}(w^k)$, and $w^k \to z$. In this way, \eqref{DRE-FBE} yields $(w^k,\alpha^k)\in\epi\big(\DREk\big)$, and thus $(z,\alpha) \in \liminn_{\gamma\downarrow0}\epi\big(\DRE\big)$. Hence, \[ \limout_{\gamma \downarrow 0}\epi\big(\DRE\big)\subseteq \epi(\varphi) \subseteq \liminn_{\gamma\downarrow0}\epi\big(\DRE\big),\] and the conclusion follows.
\end{proof}

Epi-approximation of the objective function via envelopes complements the set of relationships in \eqref{DRE:properties}. More precisely, the handy characterization of epi-convergence provided in \cite[Proposition 7.2]{rockafellar2009variational} guarantees that for any $\bar x \in \argmin \varphi$,  any approximation $x^k \to \bar x$ with $\gamma_k \downarrow 0$, $\FBEk$ and $\DREk$ approximate $\inf \varphi$ from above: \begin{equation*}
	\displaystyle\liminf_{k \to +\infty} \FBEk(x^k) \ge \varphi(\bar x), \mbox{ and } \displaystyle\liminf_{k \to +\infty} \DREk(x^k) \ge \varphi(\bar x).
\end{equation*} Furthermore, it also certifies the existence of a sequence $y^k \to \bar x$ that approximates $\inf \varphi$ from below: \begin{equation*}
	\displaystyle\limsup_{k \to +\infty} \FBEk(y^k) \le \varphi(\bar x), \mbox{ and } \displaystyle\limsup_{k \to +\infty} \DREk(y^k) \le \varphi(\bar x).
\end{equation*} Therefore, for such approximation, \begin{equation*}
	\displaystyle\lim_{k \to +\infty} \FBEk(y^k) = \displaystyle\lim_{k \to +\infty} \DREk(y^k) = \inf \varphi.
\end{equation*} 

Another consequence of Theorem~\ref{FBE-epi-convergence} and Corollary~\ref{DRE-epi-convergence} is that subgradients of $\varphi$ can be approximated using gradients of the envelopes, under further regularity assumptions. We follow the approach in \cite[Theorem 5.5]{aragon2023coderivative} to first characterize the subdifferentials of the FBE and the DRE.

\begin{lemma} \label{subdiff-FBE}
	Suppose that $f: \R^d \to \R$ is a twice continuously differentiable $L_f$-smooth function with  Lipschitz continuous Hessian, and $g: \R^d \to \R \cup \{+\infty\}$ is a proper lsc prox-bounded function with threshold $\gamma_g >0$. Then, for any $\gamma \in (0, \min\{L_f^{-1},\gamma_g\})$, \begin{equation*}
		\partial \FBE(x) = \gamma^{-1}\big( I - \gamma \nabla^2 f(x) \big)\Big[x - \mbox{conv}\Big(\prox_{\gamma g}\big(x - \gamma \nabla f(x)\big)\Big)\Big]. 
	\end{equation*}
	
\end{lemma}

\begin{proof}
	In view of \cite[eq. (58)]{aragon2023coderivative}, for all $x \in \R^d$, \begin{equation*}
		\FBE(x) + A_{\gamma g}\big(x - \gamma \nabla f(x)\big) = f(x) + \dfrac{1}{2\gamma}\|x\|^2 - \langle \nabla f(x), x\rangle, 
	\end{equation*} where $A_{\gamma g}$ is the Asplund function of $g$ with parameter $\gamma$, a convex and globally Lipschitz function whenever $\gamma \in (0, \gamma_g)$  \cite[Proposition 2.7]{aragon2023coderivative}. Since $\FBE$ is a locally Lipschitz function and $\nabla f$ is Lipschitz continuous, then $\partial^{\infty} \FBE (x) = \{0\}$ and $\partial^{\infty} \big( A_{\gamma g}(\cdot - \gamma \nabla f(\cdot)) \big) (x) = \{0\}$, and thus the transversality condition $\partial^{\infty} \FBE (x) \cap  - \partial^{\infty} \big( A_{\gamma g}(\cdot - \gamma \nabla f(\cdot)) \big) (x) = \{0\}$ trivially holds. Therefore, the sum rule for subdifferentials, \cite[Exercise 8.8]{rockafellar2009variational} and \cite[Lemma 1]{atenas2024avoids} give \begin{equation*}
		\partial \FBE(x) + \partial \big(A_{\gamma g}(\cdot - \gamma \nabla f(\cdot))\big)(x) = \gamma^{-1}\big( I - \gamma \nabla^2 f(x) \big)x. 
	\end{equation*} Furthermore, due to \cite[Theorem 10.6]{rockafellar2009variational} and \cite[Proposition 2.7]{aragon2023coderivative}, \begin{equation*}
		\partial \big(A_{\gamma g}(\cdot - \gamma \nabla f(\cdot))\big)(x) = \gamma^{-1}(I - \gamma \nabla^2 f(x)) \mbox{conv}\Big(\prox_{\gamma g}\big(x - \gamma \nabla f(x)\big)\Big).
	\end{equation*} Combining the last two equations gives the desired result.
\end{proof}

An expression for the subdifferential of $\DRE$ can be deduced then by resorting to the relationship between the FBE and the DRE in \eqref{DRE-FBE}.

\begin{corollary} \label{subdiff-DRE}
	Suppose that $f: \R^d \to \R$ is a twice continuously differentiable $L_f$-smooth function with  Lipschitz continuous Hessian, and $g: \R^d \to \R \cup \{+\infty\}$ is a proper lsc prox-bounded function with threshold $\gamma_g >0$. Then, for any $ \gamma \in (0, \min\{L_f^{-1},\gamma_g\})$, \begin{equation*}
		\partial \DRE(z) = \gamma^{-1} \nabla \prox_{\gamma f}(z) \big( I - \gamma \nabla^2 f(\prox_{\gamma f}(z))\big) \big( \prox_{\gamma f}(z) - R_{\gamma}(z) \big),
	\end{equation*} where $ R_{\gamma}(z)=\mbox{conv}\Big[\prox_{\gamma g}\big(\prox_{\gamma f}(z) - \gamma \nabla f(\prox_{\gamma f}(z))\big)\Big]$.
\end{corollary}

\begin{proof}
	Since $A_{\gamma g}$ is convex, then it is subdifferentiable regular. Applying the chain rule  \cite[Theorem 10.6]{rockafellar2009variational} to the identity \eqref{DRE-FBE} to conclude by using Lemma~\ref{subdiff-FBE} and \cite[Lemma 1]{atenas2024avoids}
\end{proof}

Under the assumptions of Lemma~\ref{subdiff-FBE} and Corollary~\ref{subdiff-DRE}, if in addition we assume that $g$ is weakly convex, then the FBE and the DRE are continuously differentiable. Indeed, for all sufficiently small $\gamma >0$, for all $x \in \R^d$, \begin{equation*}
	\nabla \FBE(x) = \gamma^{-1}\big( I - \gamma \nabla^2 f(x) \big)\Big[x - \prox_{\gamma g}\big(x - \gamma \nabla f(x)\big)\Big], 
\end{equation*} and for any $z \in \R^d$, $\nabla \DRE(z)$ takes the form \begin{equation*}
	 \gamma^{-1}\nabla \prox_{\gamma f}(z)\Big( I - \gamma \nabla^2 f\big(\prox_{\gamma f}(z)\big) \Big)\Big[\prox_{\gamma f}(z) - R_{\gamma}(z)\Big]
\end{equation*} where $ R_{\gamma}(z) = \prox_{\gamma g}\big(\prox_{\gamma f}(z) - \gamma \nabla f(\prox_{\gamma f}(z))\big)$. Furthermore, these gradients provide and \emph{outer} approximation of the subdifferential of $\varphi$. In fact, in view of \cite[Corollary 8.47]{rockafellar2009variational}, Theorem~\ref{FBE-epi-convergence} and Corollary~\ref{DRE-epi-convergence}, for any $x \in \dom(\varphi) $ and $v \in \partial \varphi (x)$, there exists a sequence $\{y^k\}$ such that, up to a subsequence, \begin{equation*}
	y^k \to x, \FBEk(y^k) \to \varphi(x), \mbox{ and } \nabla \FBEk(y^k)  \to v,
\end{equation*} and, analogously,  there exists a sequence $\{x^k\}$ such that, up to a subsequence, \begin{equation*}
	x^k \to x, \DREk(x^k) \to \varphi(x), \mbox{ and } \nabla \DREk(x^k)  \to v.
\end{equation*}

In this section, we have examined variational analysis properties of the proximal-type envelopes defined for the FB and DR methods. In the following section, we focus on algorithmic consequences of the properties of these envelopes.


\section{Convergence of Douglas-Rachford splitting through envelopes} \label{s:conv-DR}

 Independently of Section~\ref{s:epi-approx}, we now proceed to examine envelopes from a different perspective, and focus on the behavior of the iterates generated by the DR method. Different from the PPA and the FB splitting method, the sequence generated by the DR splitting method does not define a monotone sequence of function values. In this regard, the DRE arises as a device that allows to analyze the DR method using arguments of descent methods.

As mentioned in \cite[Remark 3.1]{themelis2020douglas},  if $f: \R^d\to \R$ is $L_f$-smooth and $g: \R^d\to \R\cup\{+\infty\}$ is proper lsc, then  the scheme in  \eqref{DR:scheme} is
well-defined for any $0 < \gamma < \frac{1}{L_f}$, as long as problem \eqref{problem:primal-DRS} has a nonempty set of solutions. Under the same assumptions, for the sequence $\{(x^k, y^k, z^k)\}$ generated by  \eqref{DR:scheme}, it also holds \begin{equation} \label{DRE:sandwich-i}
	\DRE(z^k) \le \varphi (x^k),
\end{equation} \begin{equation} \label{DRE:sandwich-ii}
	\varphi(y^k) \le \DRE(z^k) - \dfrac{1-\gamma L_f}{2\gamma} \|x^k - y^k\|^2.
\end{equation} As a consequence, any limit point $(x^{\star}, y^{\star}, z^{\star})$ of the sequence $\{(x^k, y^k,
z^k)\}$, whenever they exist, satisfy 
\begin{equation} \label{DRE:sandwich-iii}\DRE(z^{\star}) \le \varphi(x^{\star}), \quad \mbox{ and }
	\quad \varphi(y^{\star}) \le \DRE(z^{\star}) - \dfrac{1-\gamma
		L_f}{2\gamma}\|x^{\star}-y^{\star}\|^2. \end{equation}

\subsection{Douglas-Rachford splitting as a descent method}

Convergence properties for the DR splitting method can be obtained using customary arguments of descent methods and the DRE. The authors in \cite{themelis2020douglas} construct the tools to employ the line of reasoning detailed in 
\cite{attouch2013convergence, atenas2023unified}. The first main ingredient is to prove that the DR splitting method is a descent method for the DRE, which means $\{\DRE(z^k)\}$ satisfies a sufficient descent estimate  \cite[Theorem
4.1]{themelis2020douglas}. The second ingredient is an estimate for a subgradient of the Augmented Lagrangian in \eqref{DR:Lagrangian}, briefly mentioned in \cite{themelis2020douglas}, that let us understand DR splitting as a descent method for the DRE in the sense of \cite{atenas2023unified}. The next result summarizes these two key properties.

\begin{proposition}[Descent properties of DR splitting] \label{DR:descent} 
	
	Suppose $f: \R^d \to \R $ is a $L_f$-smooth function, $g: \R^d \to \R \cup \{+\infty\}$ is a proper lsc prox-bounded function, and for $\varphi = f + g$, $\argmin \varphi \neq \emptyset$. For $\lambda \in (0,2)$, and $ \gamma \in \left(0, \frac{2-\lambda}{2L_f}\right)$, consider the iterates $\{(x^k, y^k, z^k)\}$ generated by
	\eqref{DR:scheme}. Then, for all $k \ge 1$,    \begin{equation} \label{DRE-decrease}
		\DRE(z^k)\ge \DRE(z^{k+1}) + c\max\left\{ \frac{1}{(1+ \gamma
			L_f)^2}\|z^k -z^{k+1}\|^2, \|x^k -x^{k+1}\|^2 \right\}, \end{equation} where \begin{equation*} c = \frac{2-\lambda}{2\lambda\gamma} -
		\frac{L_f}{\lambda}>0.  \end{equation*} Furthermore, for $\x^k = \bigl(x^k, y^k, \gamma^{-1}(x^k-z^k)\bigr)$, it holds
	\begin{equation} \label{subgrad:Lagrangian} s^k := \bigl(\gamma^{-1}(x^k-y^k), 0
		,x^k-y^k\bigr) \in \hat{\partial} \mathcal{L}_{\gamma^{-1}}(\x^k), \end{equation} and
	\begin{equation} \label{subgrad:aux} \|s^k\| =
		\lambda^{-1}\sqrt{\gamma^{-2}+1}\| z^{k+1} - z^k\|.  \end{equation}
	
\end{proposition}

\begin{proof}
	
	First,  the estimate \eqref{DRE-decrease} for $\|z^k - z^{k+1}\|^2$
	corresponds to \cite[(4.2)]{themelis2020douglas} for $\sigma_{\varphi_1}= -L$. The estimate for $\|x^k
	- x^{k+1}\|^2$ appears in the proof of \cite[Theorem
	4.1]{themelis2020douglas}. For the second result, the subdifferential of $\mathcal{L}_{\gamma^{-1}}$ at $\x = \x^k$
	can be computed taking partial derivatives with respect to the different
	components of the primal-dual vector $\x$, as follows:
	\begin{itemize} 
		\item Since $f$ is differentiable, $\hat{\partial}_x\mathcal{L}_{\gamma^{-1}}(x,y,w) =
		\{\nabla f(x) + w + \gamma^{-1}(x-y)\}$. Then, taking $(x,y,w) = \x^k$, and using \eqref{u-OC}, it follows that\begin{equation*} \hat{\partial}_x\mathcal{L}_{\gamma^{-1}}(\x^k) = \{\nabla f(x^k) + 	\gamma^{-1}(x^k - z^k) + \gamma^{-1}(x^k - y^k)\}  = \{\gamma^{-1}(x^k
			- y^k)\}.\end{equation*}
		\item From the optimality condition of $y^k$ for problem \eqref{DRE}, $0 \in \hat{\partial} g(y^k) + \nabla f
		(x^k) + \gamma^{-1}(y^k - x^k)$, and thus it follows that $\partial_y \mathcal{L}_{\gamma^{-1}}(\x^k) = \hat{\partial} g (y^k) - \gamma^{-1}(x^k-z^k) + \gamma^{-1}(y^k - x^k) \ni 0$.
		\item Since $\mathcal{L}_{\gamma^{-1}}$ only depends on $w$ linearly, then  $\hat{\partial}_w \mathcal{L}_{\gamma^{-1}}(\x^k) = \{\gamma^{-1}(x^k - z^k)\}$.

	\end{itemize}
	
	Therefore, from $\hat{\partial} \mathcal{L}_{\gamma^{-1}}(\x^k) = \hat{\partial}_x
	\mathcal{L}_{\gamma^{-1}}(\x^k) \times \hat{\partial}_y
	\mathcal{L}_{\gamma^{-1}}(\x^k) \times \hat{\partial}_w
	\mathcal{L}_{\gamma^{-1}}(\x^k)$, identity \eqref{subgrad:Lagrangian} follows. To prove \eqref{subgrad:aux},
	note that due to the update rule for $\{z^k\}$ in \eqref{DR:scheme}, it
	follows \[\begin{array}{rcl} \|s^k \|^2 & = & \gamma^{-2}\|x^k - y^k\|^2 +
		\|x^k - y^k\|^2   \\ & = & (\gamma^{-2}+1)\|x^k - y^k\|^2\\ & = &
		(\gamma^{-2}+1)\lambda^{-2}\|z^k - z^{k+1}\|^2.  \end{array}\] \end{proof}

The following result is not  a straightforward application of the general scheme in \cite{atenas2023unified}. 
As made clear in the proof, in our setting, the DRE functional decrease is measured only in terms of some components of the norm of the primal-dual term $\|\x^{k}-\x^{k+1}\|^2$.
This feature prevents us to directly apply the unifying convergence theory of
\cite{atenas2023unified}.

The next result states subsequential convergence of DR splitting to critical points of $\varphi$, retrieving \cite[Theorem
4.3]{themelis2020douglas}. Here, we 
provide an alternative proof, based on the developments in \cite{atenas2023unified}. This shows that it is possible to see DR splitting as a method of descent through appropriate lenses.

\begin{theorem}[Subsequential convergence of DR splitting] \label{DR:convergence}
	
	Suppose $f: \R^d \to \R $ is a $L_f$-smooth function, $g: \R^d \to \R \cup \{+\infty\}$ is a proper lsc prox-bounded function, and for $\varphi = f + g$, $\argmin \varphi \neq \emptyset$. For $\lambda \in (0,2)$, and $ \gamma \in \left(0, \frac{2-\lambda}{2L_f}\right)$. Then any bounded sequence $\{(x^k, y^k, z^k)\}$
	generated by \eqref{DR:scheme} satisfies, 
	
	\begin{itemize} \item[(i)] The sequence $\{{\DRE(z^k)}\}$
		monotonically converges to a critical value $\varphi^{\star}$ of $\varphi$, and the sequence $\{f(x^k) + g(y^k)\}$ converges to the same value $\varphi^{\star}.$
		\item[(ii)] $x^k - y^k \to 0$, $x^k - x^{k+1} \to 0$, $y^k - y^{k+1} \to 0$, and $z^k - z^{k+1} \to 0$, as $k \to + \infty$. \item[(iii)] All cluster
		points of $\{x^k\}$ and $\{y^k\}$ coincide,
		and are also critical points of $\varphi$, with same critical value
		$\varphi^{\star}=\displaystyle\lim_{k\to \infty} \DRE(z^k)=\displaystyle\lim_{k\to \infty} f(x^k) + g(y^k)$.  \end{itemize}
	
\end{theorem}

\begin{proof}
	
	Since $\varphi$ is
	bounded from below, then  \eqref{DRE:properties} implies $\DRE$
	is also bounded from below, and thus  so is the sequence
	$\{{\DRE(z^k)}\}$. Furthermore, the descent condition
	\eqref{DRE-decrease} implies that $\{{\DRE(z^k)}\}$ is a
	nonincreasing real sequence. Thus, there exists $\varphi^{\star} \in \R$ such
	that  ${\DRE(z^k)} \to \varphi^{\star}.$ In turn,
	\eqref{DRE-decrease} then yields $z^k - z^{k+1} \to 0,$ and $x^k - x^{k+1}
	\to 0$. Therefore,  $x^k - y^k \to 0$, due to the update rule for $\{z^k\}$
	in \eqref{DR:scheme}, and thus $\{x^k\}$ and $\{y^k\}$ have the same limit
	points.  Moreover, from \eqref{subgrad:aux}, it follows that $s^k \to 0$, and  $y^k - y^{k+1} = y^k - x^k + x^k - x^{k+1} + y^{k+1} - y^{k+1}\to 0$.  This proves item (i). As for item (ii), let $x^{ \star}, y^{\star},$ and $z^{\star}$ be accumulation points
	of the sequences $\{x^k\}, \{y^k\}$, and $\{z^k\}$, respectively. Note that
	$y^{\star} = x^{\star}$,  $\varphi(y^{\star}) = \varphi(x^{\star})$, and up to a subsequence, $y^k
	\to y^{\star}$. Following the arguments in \cite{themelis2020douglas} and using \eqref{DRE:sandwich-ii}, \eqref{DRE:sandwich-i} and \eqref{DRE:sandwich-iii}, it follows that \[\varphi(y^{\star}) \le   \displaystyle\liminf_{k \to +\infty} \varphi(y^k) \le \displaystyle\limsup_{k\to +\infty } \varphi(y^k) \le  \displaystyle\limsup_{k\to +\infty } \DRE(z^k) = \DRE(z^{\star}) \le 
	\varphi(x^{\star}) =  \varphi(y^{\star}). \]
	Therefore, $\varphi(y^k) \to \varphi(y^{\star})$, and $\mathcal{L}_{\gamma^{-1}}(\x^k) = \DRE(z^k)
	\to \varphi(y^{\star}) $, with
	$\varphi(x^{\star})=\varphi(y^{\star})= \DRE(z^{\star}) = \varphi^{\star}$. Note that since $\{x^k - z^k\}$ is bounded, and $x^k - y^k \to 0$, then $\langle \gamma^{-1}(x^k- z^k), x^k - y^k \rangle + \frac{1}{2\gamma}\|x^k - y^k\|^2 \to 0$, and thus\[f(x^k) + g(y^k) = \mathcal{L}_{\gamma^{-1}}(\x^k) - \left( \langle \gamma^{-1}(x^k- y^k), x^k - y^k \rangle + \frac{1}{2\gamma}\|x^k - y^k\|^2\right) \to \varphi^{\star}.\] Furthermore, from the definition of the Augmented Lagrangian, $\mathcal{L}_{\gamma^{-1}}(x^{\star}, x^{\star},
	\gamma^{-1}(x^{\star}-z^{\star})) =  \varphi(x^{\star})$,  that is, $\{\mathcal{L}_{\gamma^{-1}}(\x^k)\}$ 
	converges to $\mathcal{L}_{\gamma^{-1}}(x^{\star}, x^{\star},
	\gamma^{-1}(x^{\star}-z^{\star}))$, as $\x^k \to (x^{\star}, x^{\star},
	\gamma^{-1}(x^{\star}-z^{\star}))$. Therefore, taking the limit in
	\eqref{subgrad:Lagrangian} (passing through a subsequence if necessary) yields
	\[0 \in \partial\mathcal{L}_{\gamma^{-1}}(x^{ \star}, x^{\star},
	\gamma^{-1}(x^{\star}- z^{\star})),\] which is equivalent to the following
	criticality conditions \begin{equation*} \left\{\begin{aligned} 0  &=  \nabla
			f(x^{\star}) + \gamma^{-1}(x^{\star}-z^{\star}) +
			\gamma^{-1}(x^{\star}-y^{\star})& \\ 0  &\in  \partial g(y^{\star}) -
			\gamma^{-1}(x^{\star}-z^{\star}) + \gamma^{-1}(x^{\star}-y^{\star})& \\ 0 & =
			y^{\star}- x^{\star}& \end{aligned}\right.  \end{equation*} Adding the first two
	relations,  the final result follows, as
	we obtain $0 \in
	\nabla f(x^{\star}) + \partial g
	(x^{\star})$.

\end{proof}

\begin{remark} Some comments about Theorem~\ref{DR:convergence} are in order.
	\begin{itemize}
		\item By items (i) and (ii), the sequence of DRE functional values
		$\DRE(z^k)$ converges monotonically to a critical value of $\varphi$. 
		By contrast, the sequence of functional values
		$f(x^k) + g(y^k)$ converges to the same critical value, but not necessarily in a monotone manner. 
		
		\item Boundedness of the iterates $\{(x^k, y^k, z^k)\}$ generated by
		\eqref{DR:scheme} can be ensured by assuming that $\varphi$ has bounded level sets \cite[Theorem 4.3(iii)]{themelis2020douglas}, which is equivalent to $\DRE$ having the same property \cite[Theorem 3.4(iii)]{themelis2020douglas}.
		
	\end{itemize}
\end{remark}

Before analyzing the rate of convergence of the DR splitting method, it is worth mentioning that it is possible to replace \eqref{subgrad:Lagrangian}-\eqref{subgrad:aux} with a bound for a subgradient of $\DRE$, by the expense of requiring \emph{more smoothness} for $f$, as in Lemma~\ref{subdiff-FBE} and Corollary~\ref{subdiff-DRE}. This corresponds to a slightly more restrictive alternative that does not require the use of the Augmented Lagrangian, and fits the setting in \cite{atenas2023unified, attouch2013convergence}. We start by first studying the FBE, and then extend the results to the DRE.

\begin{proposition}
	Suppose that $f: \R^d \to \R$ is a twice continuously differentiable $L_f$-smooth function with  Lipschitz continuous Hessian, and $g: \R^d \to \R \cup \{+\infty\}$ is a proper lsc prox-bounded function with threshold $\gamma_g >0$. Consider a bounded sequence $\{(z^k, x^k, y^k)\}$ generated by \eqref{DR:scheme} for some $\gamma \in (0, \min\{L_f^{-1},\gamma_g\})$. Then, there exists $K>0$, such that for all $k$,
	\begin{equation*}
		w^k := \gamma^{-1} \nabla \prox_{\gamma f}(z^k) \big( I - \gamma \nabla^2 f(x^k)\big) \big( x^k - y^k \big) \in \partial \DRE(z^k),
	\end{equation*} and $\|w^k\| \le K \|z^k - z^{k+1}\|$.
	
\end{proposition}

\begin{proof}
	From \eqref{DR:scheme} and Corollary~\ref{subdiff-DRE}, it follows that $y^k \in R_{\gamma}(z^k)$, and $w^k \in \partial \DRE(z^k)$. Furthermore, the boundedness assumption and \cite[Lemma 1]{atenas2024avoids} imply that $\{\nabla \prox_{\gamma f}(z^k) \big( I - \gamma \nabla^2 f(x^k)\big)\}$ is bounded. Therefore, the bound for $\{w^k\}$ can be thus deduced from the third step in \eqref{DR:scheme}.
\end{proof}

Although this result would allow following \cite{atenas2023unified, attouch2013convergence} to obtain subsequential convergence, in the next section we do not assume that $f$ is twice continuously differentiable, because $L_f$-smoothness is enough, as shown in Theorem~\ref{DR:convergence}.

\subsection{Rate of convergence of nonconvex Douglas-Rachford splitting}

The analysis for global convergence and local rates of convergence requires additional regularity assumptions for problem \eqref{problem:primal-DRS}. This
section is an extension of \cite{themelis2020douglas},
by applying the machinery of \cite{atenas2023unified} to \eqref{DR:scheme} through
the envelope
$\DRE$ and the augmented Lagrangian $\mathcal{L}_{\gamma^{-1}}.$ Recall that in this section, we do not assume that $f$ is twice continuously differentiable.

We say a function $\varphi: \R^n \to \R\cup \{+\infty\}$ satisfies a local error bound, if for any $\bar{\varphi} \ge \inf \varphi > -\infty$, there exist constants $\varepsilon , \ell >0,$ such that whenever $\varphi(y) \le \bar{\varphi}$,
\begin{equation} \label{EB-f}  \dist\bigl(y, (\partial
	\varphi)^{-1}(0)\bigr) \le \ell \dist\bigl(0, \partial \varphi (y) \cap B(0,
	\varepsilon)\bigr).  \end{equation}

Another assumption we need for the analysis of rates of convergence is the following. We say a function $\varphi: \R^n \to \R\cup \{+\infty\}$ satisfies the proper
separation of isocost surfaces property if there exists $\delta >0$, such that 
\begin{equation} \label{PSIS}  \forall x,y \in (\partial
	\varphi)^{-1}(0), \: \| x- y \| \le \delta \implies \varphi(x) =
	\varphi(y).  \end{equation}

For locally Lipschitz functions, the above  subdifferential-based error bound  is implied by metric subregularity of the subdifferential $\partial \varphi$ at $\bar{x} \in \argmin \varphi$ for $0$. See \cite{drusvyatskiy2015quadratic,drusvyatskiy2013tilt} for relationships between metric subregularity and quadratic growth in nonconvex settings, and \cite{artacho2013metric} in the convex case. In addition, the error bound above together with proper separation of isocost surfaces imply the so called Kurdyka-{\L}ojasiewic inequality with exponent $1/2$ \cite[Theorem 4.1]{li2018calculus} .

The next result relates the sequences generated by \eqref{DR:scheme} with the  error bound condition in \eqref{EB-f}, resulting in an estimate crucial to obtain local rate of convergence.

\begin{proposition} \label{prop:EB}  Suppose $f: \R^d \to \R $ is a $L_f$-smooth function, $g: \R^d \to \R \cup \{+\infty\}$ is a proper lsc prox-bounded function, and $\varphi = f + g$ has a nonempty set of minimizers.  In addition, assume the
	local error bound \eqref{EB-f} holds. Take any $\lambda \in (0,2)$, and $ \gamma \in (0, \frac{2-\lambda}{2L_f})$. Then, for any bounded sequence $\{(x^k, y^k, z^k)\}$
	generated by \eqref{DR:scheme}, there exist $\varepsilon>0$, $ \ell >0$, $\bar{\varphi}>0$, and a sequence $\{d^k\}$ given by \begin{equation*}
		d^k = \gamma^{-1}(x^k - y^k) - (\nabla
		f(x^k) -  \nabla f(y^k)),
	\end{equation*} such
	that $d^k \to 0$, and for all sufficiently large $k$, $d^k \in \partial \varphi (y^k) \cap B(0,\varepsilon)$,
	$\varphi(y^k) \le \bar{\varphi}$, and  \begin{equation} \label{EB:k} \dist\bigl(y^k, (\partial \varphi)^{-1}(0)\bigr) \le
		\ell \| d^k\|.  \end{equation}
	
\end{proposition}

\begin{proof}
	
	Since $\{\DRE(z^k)\}$ monotonically converges to $\varphi^{\star}$, then
	for any $\epsilon>0$, and for any sufficiently large $k,$  $\mathcal{L}_{\gamma^{-1}}(\x^k) =
	\DRE(z^k) \le \varphi^{\star} + \epsilon$, and thus \begin{equation*} 
		f(x^k) +
		g(y^k) + \gamma^{-1}\langle x^k-z^k, x^k-y^k\rangle +
		\dfrac{1}{2\gamma}\|x^k-y^k\|^2 \le \varphi^{\star} + \epsilon.  \end{equation*} Combining this inequality with Lemma~\ref{descent-lemma} and \eqref{u-OC}, we obtain \begin{equation*}
		\varphi(y^k) \le \varphi^{\star} + \epsilon + \dfrac{1}{2}\left(L_f - \dfrac{1}{\gamma}\right)\|x^k-y^k\|^2.
	\end{equation*} Since $\gamma < L_f^{-1}$, the right-most term is negative, and thus $\varphi(y^k) < \overline{\varphi}$ for $ \overline{\varphi}= \varphi^{\star} + \epsilon$.

Furthermore, from the optimality conditions of \eqref{DRE:moreau} (as shown in the proof of \cite[Theorem 4.3]{themelis2020douglas}), it follows
that $d^k \in \hat{\partial} \varphi
(y^k).$ Since $\nabla f$ is $L_f$-Lipschitz continuous,
then \begin{equation} \label{subgrad:phi} \| d^k\| \le
	\gamma^{-1}(\|x^k-y^k\| + \gamma L_f \| x^k - y^k\|) = \gamma^{-1}(1+\gamma
	L_f)\|x^k - y^k\|.  \end{equation} In this manner,  from
Theorem~\ref{DR:convergence}(ii), $d^k \to 0.$ Then, for $\varepsilon >0$ given in \eqref{EB-f}, and   all
sufficiently large $k$, $d^k \in \partial \varphi(y^k) \cap B(0,
\varepsilon),$ and $\varphi(y^k) \le \bar{\varphi}.$ This allows us to apply
\eqref{EB-f} to obtain \eqref{EB:k}.  \end{proof}

Before giving the main result of this section, first we need some
technical estimates deduced from Proposition~\ref{prop:EB}. As a consequence of the error bound, it is possible to bound the difference of $\DRE$ at two consecutive iterates using the primal-dual sequence $\{\xi^k\}$ (cf. Proposition~\ref{DR:descent}).

\begin{lemma} \label{lemma:proj}

Suppose the conditions of Proposition~\ref{prop:EB} hold. For any $p^k_y \in
\mbox{proj}_{(\partial \varphi)^{-1}(0)}(y^k),$ define \[p^k = (p_y^k, p_y^k,
\gamma^{-1}(p_y^k - z^k)). \] Then, there exists $\overline{C}>0$, such that for
all $k$,  \begin{equation} \label{pk-xk} \| p^k - \x^k\|^2 \le
	\overline{C}({\DRE(z^k) - \DRE(z^{k+1})}).  \end{equation}

\end{lemma}

\begin{proof}

From the definition of $p^k$ and $\x^k$, we have \begin{equation*} \begin{aligned} \| p^k - \x^k\|^2  &=  \| p^k_y - x^k\|^2 + \|
		p^k_y - y^k\|^2 + \|\gamma^{-1}(p_y^k - z^k) - \gamma^{-1}(x^k - z^k) \|^2& \\
		&= (1+ \gamma^{-2})  \| p^k_y - x^k\|^2 + \| p^k_y - y^k\|^2. & \end{aligned}
\end{equation*} From \eqref{EB:k}, it follows that $ \| p^k_y - y^k\| \le \ell
\| d^k\|.$ As for $\| p^k_y - x^k\|$, it holds that \begin{equation*}
	\begin{aligned} \| p^k_y - x^k\|^2  & \le  (\| p^k_y - y^k\| + \| y^k - x^k\|)^2
		& \\ &\le  2\| p^k_y - y^k\|^2 + 2\| y^k - x^k\|^2 & \\ & \le  2\ell^2 \|
		d^k\|^2 + 2\| y^k - x^k\|^2 & \\ & =  2\ell^2 \| d^k\|^2 +
		2\lambda^{-2}\| z^k -z^{k+1}\|^2 & \\ \end{aligned} \end{equation*} where for the first inequality we apply the triangle inequality, for the second
inequality we use the estimate $(a+b)^2 \le 2(a^2 + b^2)$, and for the last equality we use the update rule for $\{z^k\}$ in \eqref{DR:scheme}.
Therefore,   \begin{equation*} \| p^k - \x^k\|^2 \le 2(1+ \gamma^{-2})  (\ell^2 \|
	d^k\|^2 + \lambda^{-2}\| z^k - z^{k+1}\|^2) + \ell^2 \| d^k\|^2.
\end{equation*}

Now, we bound the terms in the right-hand side of the above estimate.  First, note that the descent condition
\eqref{DRE-decrease} implies \begin{equation} \label{s:aux} \|z^k -z^{k+1}\|^2
	\le \dfrac{(1+ \gamma L_f)^2}{c} ({\DRE(z^k) - \DRE(z^{k+1})}).
\end{equation} Secondly, \eqref{subgrad:phi} together with \eqref{DR:scheme} and
\eqref{s:aux} imply \begin{equation} \label{bound-dk}\begin{aligned} \|d^k\|^2 & =  \gamma^{-2}(1+\gamma
		L_f)^2\| x^k - y^k\|^2 & \\ &=  (\lambda\gamma)^{-2}(1+\gamma L_f)^2\| z^k -
		z^{k+1}\|^2 & \\ & \le  (\lambda\gamma)^{-2}(1+\gamma L_f)^2\dfrac{(1+ \gamma
			L_f)^2}{c} ({\DRE(z^k) - \DRE(z^{k+1})}) & \\ \end{aligned}
\end{equation} Hence,  \eqref{pk-xk}  follows for \begin{equation*} \overline{C} =
	\frac{(1+\gamma
		L_f)^2}{c\lambda^{2}}\left(\frac{\bigl(2(1+\gamma^{-2})+1\bigr)(1+\gamma L_f)^2\ell^2}{
		\gamma^2} +2(1+\gamma^{-2})\right).  \end{equation*}
		\end{proof}
		
		In order to relate an
		error bound for $\varphi$ with the descent properties of Theorem~\ref{DR:descent}, we need to assume that $g$ is $\rho$-weakly convex. In this manner, both $f$ and $g$ are weakly convex. In view of \eqref{prox-subgrad},   subgradients of weakly convex functions can be characterized as proximal subgradients \cite[Definition 8.35]{rockafellar2009variational} in the whole space. This particular form for subgradients is key for the next technical result,  showing that the augmented Lagrangian  satisfies the (proximal) subgradient inequality for weakly convex functions throughout the path defined by $\{\xi^k\}$.
		
		\begin{lemma} \label{Lagrangian-wc}
Suppose $f: \R^d \to \R $ is a $L_f$-smooth function, $g: \R^d \to \R \cup \{+\infty\}$ is a proper lsc $\rho$-weakly convex function, and $\varphi = f + g$ has a nonempty set of minimizers. Then the subgradient defined in \eqref{subgrad:Lagrangian} satisfies the following estimate for all $k$: \begin{equation*}
	\mathcal{L}_{\gamma^{-1}}(\x^k) + \langle s^k, p^k - \x^k \rangle \le \mathcal{L}_{\gamma^{-1}}(p^k) + \frac{\max\{L_f,\rho,\gamma\}}{2}\|p^k-\x^k\|^2.
\end{equation*}
\end{lemma}

\begin{proof}
From the definition of $s^k$ in \eqref{subgrad:Lagrangian}, it holds for all $k$, \begin{equation*}
	\begin{array}{ll}
		&\mathcal{L}_{\gamma^{-1}}(\x^k) + \langle d^k, p^k - \x^k \rangle  \\
		=& f(x^k) + g(y^k) + \gamma^{-1}\langle x^k-z^k, x^k-y^k \rangle + \frac{1}{2\gamma}\|x^k-y^k\|^2 \\
		& + \gamma^{-1} \langle x^k - y^k, p_y^k - x^k \rangle + \langle x^k - y^k, \gamma^{-1}(p_y^k - z^k) - \gamma^{-1}(x^k - z^k) \rangle\\
		=& f(x^k) + g(y^k) + \gamma^{-1}\langle x^k-z^k, x^k-y^k \rangle + \frac{1}{2\gamma}\|x^k-y^k\|^2  + 2\gamma^{-1} \langle x^k - y^k, p_y^k - x^k \rangle .
	\end{array}
\end{equation*} We will reduce the inner products and the squared term, by first working on them separately. Define $d_2^k = -\nabla f(x^k) - \gamma^{-1}(y^k-x^k)$, which in view of the optimality conditions of $y^k = \prox_{\gamma g}(x^k - \gamma \nabla f (x^k))$, satisfies $d_2^k \in \partial g(y^k)$. In this way, from \eqref{u-OC},\begin{equation*}
	\begin{array}{ll}
		&\gamma^{-1}\langle x^k-z^k, x^k-y^k \rangle  \\
		= &  \langle -\nabla f(x^k), x^k-y^k \rangle \\
		= &  \langle d_2^k, x^k-y^k \rangle + \gamma^{-1}\langle y^k - x^k, x^k-y^k \rangle \\
		= &  \langle d_2^k, p_y^k-y^k \rangle + \langle d_2^k, x^k-p_y^k \rangle - \gamma^{-1}\| x^k-y^k\|^2. 
	\end{array}
\end{equation*} Furthermore, \begin{equation*}
	\begin{array}{ll}
		&  2\gamma^{-1} \langle x^k - y^k, p_y^k - x^k \rangle  \\
		= & 2\gamma^{-1} \langle x^k - y^k, p_y^k - x^k \rangle  - \langle \nabla f(x^k) + \gamma^{-1}(y^k-x^k), x^k-p_y^k \rangle - \langle d_2^k, x^k-p_y^k \rangle \\
		= & 2\gamma^{-1} \langle x^k - y^k, p_y^k - x^k \rangle  + \gamma^{-1}\langle 2x^k - z^k - y^k, x^k-p_y^k \rangle - \langle d_2^k, x^k-p_y^k \rangle \\
		= &  \langle \gamma^{-1}(z^k - y^k), p_y^k - x^k \rangle  - \langle d_2^k, x^k-p_y^k \rangle \\
		= &  \langle \nabla f(x^k), p_y^k - x^k \rangle + \gamma^{ -1} \langle x^k - y^k, p_y^k - x^k \rangle - \langle d_2^k, x^k-p_y^k \rangle ,
	\end{array}
\end{equation*} where the third identity and the last line follow from \eqref{u-OC}. Therefore, gathering  the aforementioned terms, it holds that \begin{equation*}
	\begin{array}{ll}
		&  \gamma^{-1}\langle x^k-z^k, x^k-y^k \rangle + \frac{1}{2\gamma}\|x^k-y^k\|^2  + 2\gamma^{-1} \langle x^k - y^k, p_y^k - x^k \rangle  \\
		= & \langle d_2^k, p_y^k-y^k \rangle - \gamma^{-1}\| x^k-y^k\|^2 + \langle \nabla f(x^k), p_y^k - x^k \rangle + \gamma^{ -1} \langle x^k - y^k, p_y^k - x^k \rangle \\ & \quad + \frac{1}{2\gamma}\|x^k-y^k\|^2 \\
		= & \langle d_2^k, p_y^k-y^k \rangle  + \langle \nabla f(x^k), p_y^k - x^k \rangle - (\frac{1}{2\gamma}\|x^k-y^k\|^2 - \gamma^{ -1} \langle x^k - y^k, p_y^k - x^k \rangle)   \\
		= & \langle d_2^k, p_y^k-y^k \rangle  + \langle \nabla f(x^k), p_y^k - x^k \rangle - (\frac{1}{2\gamma}\|x^k-y^k -(p_y^k-x^k)\|^2 - \frac{1}{2\gamma}\|p_y^k-x^k\|^2)   \\
		\le & \langle d_2^k, p_y^k-y^k \rangle  + \langle \nabla f(x^k), p_x^k - x^k \rangle + \frac{1}{2\gamma}\|p_y^k-x^k\|^2  \\
	\end{array}
\end{equation*} Hence, Lemma~\ref{descent-lemma} and \eqref{prox-subgrad} yield, \begin{equation*}
	\begin{array}{ll}
		&\mathcal{L}_{\gamma^{-1}}(\x^k) + \langle s^k, p^k - \x^k \rangle  \\
		\le & f(x^k) +\langle \nabla f(x^k), p_y^k - x^k \rangle  + g(y^k) + \langle d_2^k, p_y^k-y^k \rangle + \frac{1}{2\gamma}\|p_y^k-x^k\|^2 \\
		\le & f(p_y^k) + \frac{L}{2}\|p_y^k - x^k\|^2 +  g(p_y^k) + \frac{\rho}{2}\|p_y^k - y^k\|^2+ \frac{\gamma}{2}\|\gamma^{-1}(p_y^k-x^k)\|^2 \\
		= & f(p_y^k) + \frac{L_f}{2}\|p_y^k - x^k\|^2 +  g(p_y^k) + \frac{\rho}{2}\|p_y^k - y^k\|^2+ \frac{\gamma}{2}\|\gamma^{-1}(p_y^k-z^k)-\gamma^{-1}(x^k-z^k)\|^2 \\
		= & \mathcal{L}_{\gamma^{-1}}(p^k) + \frac{L_f}{2}\|p_y^k - x^k\|^2 + \frac{\rho}{2}\|p_y^k - y^k\|^2+ \frac{\gamma}{2}\|\gamma^{-1}(p_y^k-z^k)-\gamma^{-1}(x^k-z^k)\|^2, 
	\end{array}
\end{equation*} from where we can conclude.
\end{proof}

The next theorem is the main result of this section, establishing local linear rates of convergence of DR splitting for weakly convex problems. 

\begin{theorem}[Rate of convergence of nonconvex DR splitting] \label{DR:rate}

Suppose $f: \R^d \to \R $ is a $L_f$-smooth function, $g: \R^d \to \R \cup \{+\infty\}$ is a proper lsc $\rho$-weakly convex function, and $\varphi = f + g$ has a nonempty set of minimizers.  In addition, assume the
local error bound \eqref{EB-f} holds, as well as the proper separation of isocost surfaces property \eqref{PSIS}. Then, for $\lambda \in (0,2)$,  $ \gamma \in (0, \frac{2-\lambda}{2L_f})$, and any bounded sequence $\{(x^k, y^k, z^k)\}$
generated by \eqref{DR:scheme}, 

\begin{itemize} \item[(i)] The sequence $\{\DRE(z^k)\}$ $Q-$linearly converges to
	a critical value $\varphi^{\star}$ of $\varphi,$ and the sequence $\{f(x^k)+g(y^k)\}$ $R-$linearly converges to the same value $\varphi^{\star}.$
	
	\item[(ii)] The sequences $\{x^k\}$ and $\{y^k\}$ $R-$linearly converge to a
	critical point $x^{\star}$ of $\varphi$, and $\{z^k\}$ $R-$linearly
	converges to a point $z^{\star}$, such that $x^{\star} = \prox_{\gamma
		f}(x^{\star})$.  \end{itemize}

\end{theorem}

\begin{proof}

First, from Proposition~\ref{prop:EB}, since $ d^k \to 0$, then $y^k - p^k_y
\to 0$, which in turn implies  $p^k_y- p^{k+1}_y \to 0$, in view of
Theorem~\ref{DR:convergence}(ii). Then, applying the proper separation of isocost
surfaces property \eqref{PSIS}, for all sufficiently large $k,$ $\varphi(p^k_y)
= \varphi(p^{k+1}_y),$ and thus $\varphi(p^k_y) = \varphi^{\star},$ for some
critical value $\varphi^{\star}$ of $\varphi$. From
Theorem~\ref{DR:convergence}(iii), up to a subsequence, $y^k \to x^{\star}$, for
a critical point $x^{\star}$ of $\varphi$. Therefore, $p_y^k \to x^{\star}$  and
$\varphi(p_y^k) \to \varphi(x^{\star})$, for the same subsequence. Hence
$\varphi^{\star} = \varphi(x^{\star})$.

Furthermore, from the definition of the augmented Lagrangian, we have
$\mathcal{L}_{\gamma^{-1}}(p^k) = \varphi(p_y^k)$. In view of
Lemma~\ref{Lagrangian-wc}, for all sufficiently large $k$,   \begin{equation*} \begin{array}{rcl} \DRE(z^k)-
		\varphi(p_y^k) & =&\mathcal{L}_{\gamma^{-1}}(\x^k)- \mathcal{L}_{\gamma^{-1}}(p^k)
		\\ & \le &-\langle d^k, p^k - \x^k \rangle + \dfrac{\bar{\rho}}{2}\| p^k - \x^k \|^2,
\end{array} \end{equation*} where $\bar{\rho} = \max\{L,\rho,\gamma\}$. Hence,  \begin{equation}
	\label{Lagr:estimate} \DRE(z^k) - \varphi^{\star}  \le \|s^k \| \| p^k - \x^k \|+
	\dfrac{\bar{\rho}}{2}\| p^k - \x^k \|^2  .  \end{equation}

Combine Proposition~\ref{DR:descent}  and \eqref{pk-xk}  
to obtain \begin{equation} \label{DL} \DRE(z^k) - \varphi^{\star}
	\le  \left(\tilde{C}+ \overline{C}\dfrac{\bar{\rho}}{2} \right)({\DRE(z^k)
		- \DRE(z^{k+1})}) \end{equation} for $\tilde{C} = \lambda^{-1}\sqrt{\gamma^{-2}+1}(1+\gamma L_f)
\sqrt{\frac{\overline{C}}{c}}$.   Set $\hat{C} := \tilde{C}+ \overline{C}\dfrac{\bar{\rho}}{2},$ $r =
\dfrac{\hat{C}}{1+\hat{C}} \in (0,1),$ and $\phi^k := \DRE(z^k) - \varphi^{\star}
$. Monotonicity of $\{\DRE(z^k)\}$ implies $\phi^{k+1} \le \phi^k$. Thus,
from \eqref{DL}, for all sufficiently large $k,$\begin{equation*} \phi^{k+1} \le
	\hat{C}(\phi^k - \phi^{k+1}) \iff \phi^{k+1} \le r \phi^k, \end{equation*} from which the first part of item (i) follows. 

For item (ii), suppose that above estimate holds for all $k \ge k_0.$ Then,
\begin{equation*} \phi^{k+1} \le (\phi^{k_0}r^{1-k_0})r^k, \end{equation*} or
equivalently, for $q = \phi^{k_0}r^{-k_0}$ and all $k \ge k_0 +1$ \begin{equation} \label{Vk:r-rate}
	\phi^k \le
	q r^k.  \end{equation} From the descent condition  of Proposition~\ref{DR:descent}, it follows \begin{equation}
	\label{descent:Vk} \|z^k - z^{k+1}\| \le \dfrac{1+\gamma L_f}{\sqrt{c}}\sqrt{\phi^k},
	\quad \mbox{ and } \quad \|x^k - x^{k+1}\| \le \dfrac{1}{\sqrt{c}} \sqrt{\phi^k}.
\end{equation} Therefore, from \cite[Lemma 4.1]{atenas2023unified}, 
there exists
$m>0$, $\alpha \in (0,1)$, $z^{\star} \in \R^d$ such that for all
sufficiently large $k$ \begin{equation*} \|z^k- z^{\star}\| \le m \alpha^k,
	\quad \|x^k -x^{\star}\| \le m \alpha^k.  \end{equation*} Note that  $\{x^k\}$ converges to the critical point $x^{\star}$, since $\{x^k\}$ and $\{y^k\}$ have the same limit points. Observe that since $f$ is $L_f$-smooth, then
$\prox_{\gamma f}$ is Lipschitz continuous \cite[Proposition
2.3(ii)]{themelis2020douglas}, therefore $x^k = \prox_{\gamma f}(z^k)
\to \prox_{\gamma f}(z^{\star})$, and $x^{\star}=\prox_{\gamma f}(z^{\star}).$

In addition, the rate of convergence of $\{y^k\}$ can be deduced using the
triangle inequality, the update rule for $\{z^k\}$ in \eqref{DR:scheme}, and
\eqref{descent:Vk} as follows:

\begin{equation*} \begin{aligned} \|y^k - y^{k+1}\|  & \le  \|y^k - x^{k}\| +
		\|x^k - x^{k+1}\| + \|x^{k+1} - y^{k+1}\|& \\ &=\lambda^{-1}\|z^k - z^{k+1}\| +
		\|x^k - x^{k+1}\| + \lambda^{-1}\|z^{k+1} - z^{k+2}\|& \\ & \le
		\lambda^{-1}\dfrac{1+\gamma L_f}{\sqrt{c}}\sqrt{\phi^k} +
		\dfrac{1}{\sqrt{c}}\sqrt{\phi^k} + \lambda^{-1}\dfrac{1+\gamma
			L_f}{\sqrt{c}}\sqrt{\phi^{k+1}}.  & \end{aligned} \end{equation*} Since $\{\phi^k\}$ is nonincreasing, then \begin{equation*} \|y^k - y^{k+1}\| \le
	\left( \dfrac{2\lambda^{-1}(1+\gamma L_f) +1}{\sqrt{c}}\right) \sqrt{\phi^k}.
\end{equation*} Then, from \cite[Lemma 4.1]{atenas2023unified} it follows that for all
sufficiently large $k$, \begin{equation*} \|y^k - y^{\star}\| \le \bar{m}
	\bar{\alpha}^k.  \end{equation*} for some $\bar{m}>0$ and $\bar{\alpha}\in
(0,1).$

Finally, to obtain the rate of convergence of $\{f(x^k)+ g(y^k)\}$, first note that since $\DRE(z^{k+1}) \ge \varphi^{\star}$, then from \eqref{DRE-decrease}, \eqref{Vk:r-rate}, and \eqref{DR:scheme}, it follows \begin{equation} \label{R-rate:aux}
	\lambda^2 \| x^k - y^k\|^2 = \|z^k - z^{k+1}\|^2 \le \left[ \frac{(1+\gamma L_f)^2 q}{c} \right] r^k.
\end{equation} Furthermore, in view of \eqref{DR:Lagrangian}, \begin{equation*}
	|f(x^k) + g(y^k) - \varphi^{\star}| \le \DRE(z^k)-\varphi^{\star} + \frac{1}{\gamma}\|x^k-z^k\| \|x^k - y^k\| + \frac{1}{2 \gamma }\|x^k - y^k\|^2.
\end{equation*} By assumption, $\{x^k\}$ and $\{z^k\}$ are bounded sequences, therefore there exists $M_2>0$, such that for all $k$, $\|x^k- z^k\| \le M_2.$ Substituting this estimate, \eqref{Vk:r-rate} and \eqref{R-rate:aux} in the above inequality, yields \begin{equation*}
	|f(x^k) + g(y^k) - \varphi^{\star}| \le qr^k + \left( \frac{M_2(1+\gamma L_f)}{\gamma \lambda} \sqrt{\frac{q}{c}}\right) \sqrt{r}^k + \left( \frac{(1+\gamma L_f)^2 q}{2 \gamma c \lambda^2} \right) r^k.
\end{equation*} Since $r\in (0,1)$, then $r\le \sqrt{r}$, and thus \begin{equation*}
	|f(x^k) + g(y^k) - \varphi^{\star}| \le \tilde{K} \sqrt{r}^k,
\end{equation*} where \begin{equation*}
	\tilde{K} = q +  \frac{M_2(1+\gamma L_f)}{\gamma \lambda} \sqrt{\frac{q}{c}} + \frac{(1+\gamma L_f)^2 q}{2 \gamma c \lambda^2}. 
\end{equation*} This proves the second part of item (i).

\end{proof}

\begin{remark}
    The authors in \cite[Theorem 3]{li2016douglas} established rates of convergence for the subgradient smallest norm in a different but related nonconvex setting, namely, for semialgebraic optimization. Observe that Theorem~\ref{DR:rate} goes a bit further and shows rates of convergence for function values and for iterates of DR. Similar results to Theorem~\ref{DR:rate} can be obtained from \cite[Theorem 4.4]{themelis2020douglas} under the assumption of the Kurdyka-{\L}ojasiewic inequality, using arguably standard arguments.
\end{remark}

In this section, we show that the convergence results of the DR splitting method in the convex case can be extended to the weakly convex setting under appropriate regularity assumptions, and by resorting to the DRE. In principle, as long as it is possible to define a Moreau-type envelope for a nonmonotone method, such as DR, the same line of reasoning can be applied, which resembles the one employed in \cite{atenas2023unified} for descent methods.


\section{Numerical experiments} \label{section:numerical}

In this section, we present the numerical performance of the DR splitting method applied to the regularized least squares problem with nonconvex regularizers. The role of regularization is to induce sparsity in the solution via approximation of the $\ell_0$ pseudonorm, and one way to achieve that is to use penalties that are sharp around the origin, such as the $\ell_1$ norm \cite{santosa1986linear}.

In our experiments, we analyze two nonconvex penalties. First, we consider the \emph{minimax convex penalty} (MCP) defined as follows:  for $x \in \R$, and $\sigma, \theta >0$,  $g(x) = \sigma |x| - \frac{1}{2\theta}x^2 $ if $|x| \le \theta\sigma$, and $g(x) = \frac{\theta\sigma^2}{2}$ otherwise. This function is $\theta^{-1}$-weakly convex \cite{bohm2021variable}, and its proximal operator, called the \emph{firm threshold} \cite{gao1997waveshrink}, is given by \begin{equation} \label{prox-MCP}
	\prox_{\gamma g}(x) = \left\{\begin{array}{cl}
		0 & \mbox{ if } |x| < \gamma\sigma \\
		\dfrac{x - \sigma \gamma \mbox{sign}(x)}{1-\frac{\gamma}{\theta}} & \mbox{ if } \gamma\sigma \le |x| \le \theta\sigma\\
		x & \mbox{ if }|x| > \theta\sigma.
	\end{array}\right.
\end{equation}  whenever $\theta > \gamma$. Secondly, we examine the \emph{smoothly clipped absolute deviation} (SCAD) penalty, defined for $x \in \R$ and $\sigma >0$ and $\theta >2$ as: $g(x) = \sigma |x|$ if $|x| < \sigma$, $g(x) = \frac{-x^2 + 2\theta\sigma|x| - \sigma^2}{2(\theta-1)}$ if $\sigma \le |x| \le \theta\sigma $, and $g(x) = \frac{(\theta+1)\sigma^2}{2}$ if $|x| > \theta\sigma $. This function is $(\theta-1)^{-1}$-weakly convex \cite{bohm2021variable}, and its proximal operator is given by \cite{fan2001variable} \begin{equation} \label{prox-SCAD}
	\prox_{\gamma g}(x) = \left\{\begin{array}{cl}
		\mbox{sign}(x)\max\{0, |x| - \sigma\} & \mbox{ if } |x| < \sigma (\gamma + 1 ) \\
		\dfrac{(\theta-1)x - \theta\sigma \gamma \mbox{sign}(x)}{\theta - \gamma - 1 } & \mbox{ if } \sigma (\gamma + 1 ) \le |x| \le \theta\sigma\\
		x & \mbox{ if }|x| > \theta\sigma.
	\end{array}\right.
\end{equation} The classical approach of using the $\ell_1$ norm as penalty function has the disadvantage of introducing bias, since the proximal operator of the $\ell_1$ norm, the \emph{soft thresholding} operator, does not approach the identity for arguments with large magnitude, while both \eqref{prox-MCP} and \eqref{prox-SCAD} both do. 

The battery of problems we address have the structure in \eqref{problem:primal-DRS}, with $f(x) = \frac{1}{2}\|Ax - b\|_2^2$, where $A$ is a convolution matrix of dimension $m \times n$, \begin{equation*}
	(m, n) \in \{(10,30), (30, 90), (50, 150)\},
\end{equation*} and $b$ is a $m$-dimensional vector generated as follows: given a sparse randomly generated target point $\bar x \in \R^n$, define $b = A \bar{x} + e$, where $e$ is a random noise vector. Regarding the regularization component, $g$ takes the form of the MCP and SCAD penalties. Clearly, $f$ is a continuously differentiable function with Lipschitz continuous gradient with Lipschitz constant $L_f = \|A^{\top}A\|_2$, where $\|\cdot\|_2$  denotes the spectral norm. We use  $\lambda \in \{0.5, 1.0, 1.5\} $ for the over-relaxation parameter, and $\gamma = \alpha \cdot \frac{2-\lambda}{2 L_f}$ for the proximal parameter with $\alpha \in \{0.3, 0.5, 0.9\}$. For both nonconvex regularizers, we set $\sigma \in \{1, 10^{-1}, 10^{-2}\}$,   and for penalty-dependent parameters,  we use $\theta = 1.5 \gamma$ for MCP, and $\theta = 1.5 (\gamma + 1)$ for SCAD. We also set $5000$ as the maximum number of iterations, and  $ \|x^k - y^k \| \le 10^{-6}$ as stopping test, based on Proposition~\ref{DR:descent} and Corollary~\ref{subdiff-DRE}: a small value of $\|x^k - y^k\|$ yields an \emph{approximate} critical point. 

\begin{figure}[b!]
	\centering
	\begin{subfigure}{.475\textwidth}
		\centering
		\includegraphics[width=.95\linewidth]{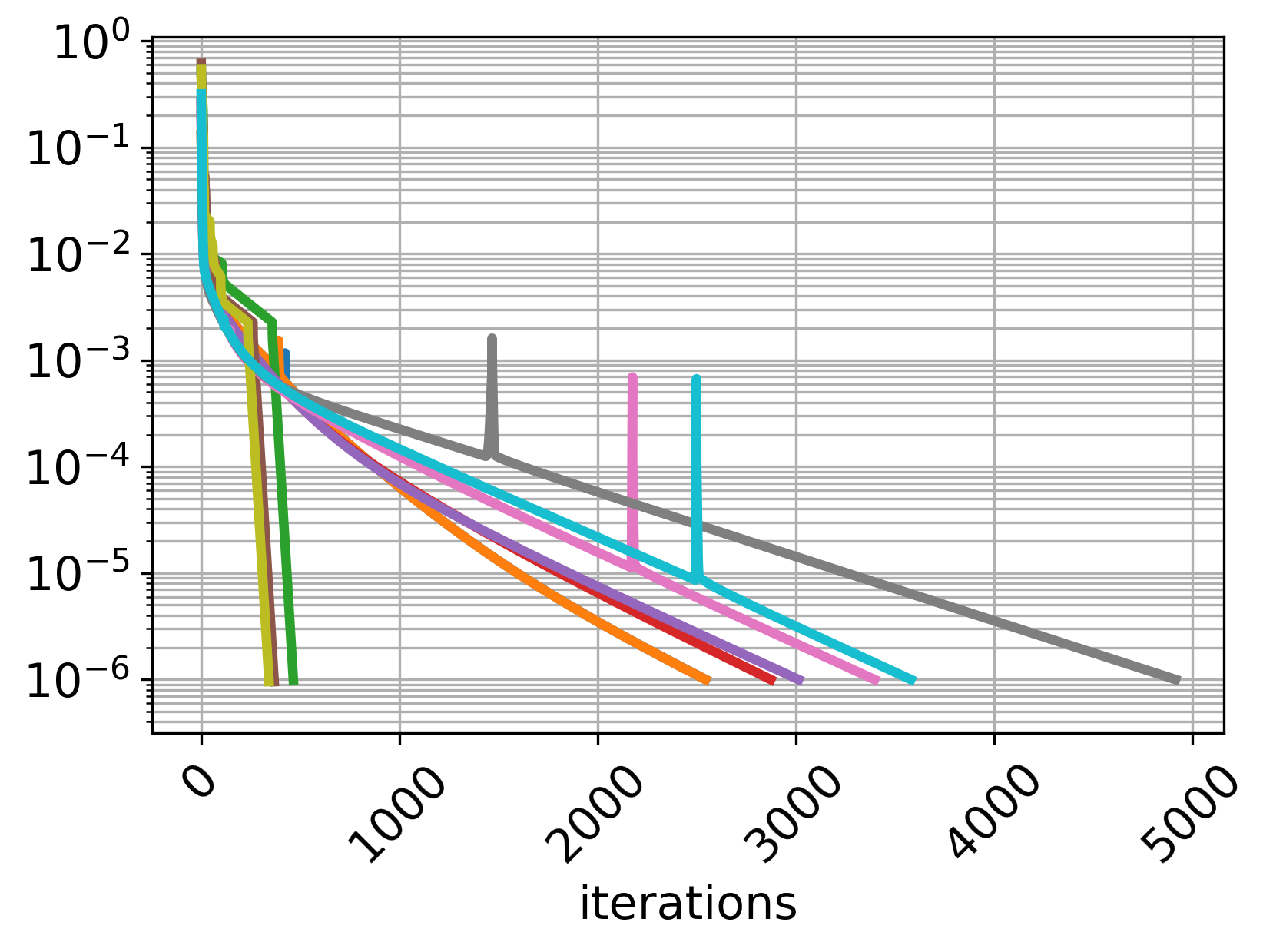}
		\caption{$\{\|x^k-y^k\|\}$}
		\label{fig:sub1a}
	\end{subfigure}%
	\begin{subfigure}{.475\textwidth}
		\centering
		\includegraphics[width=.95\linewidth]{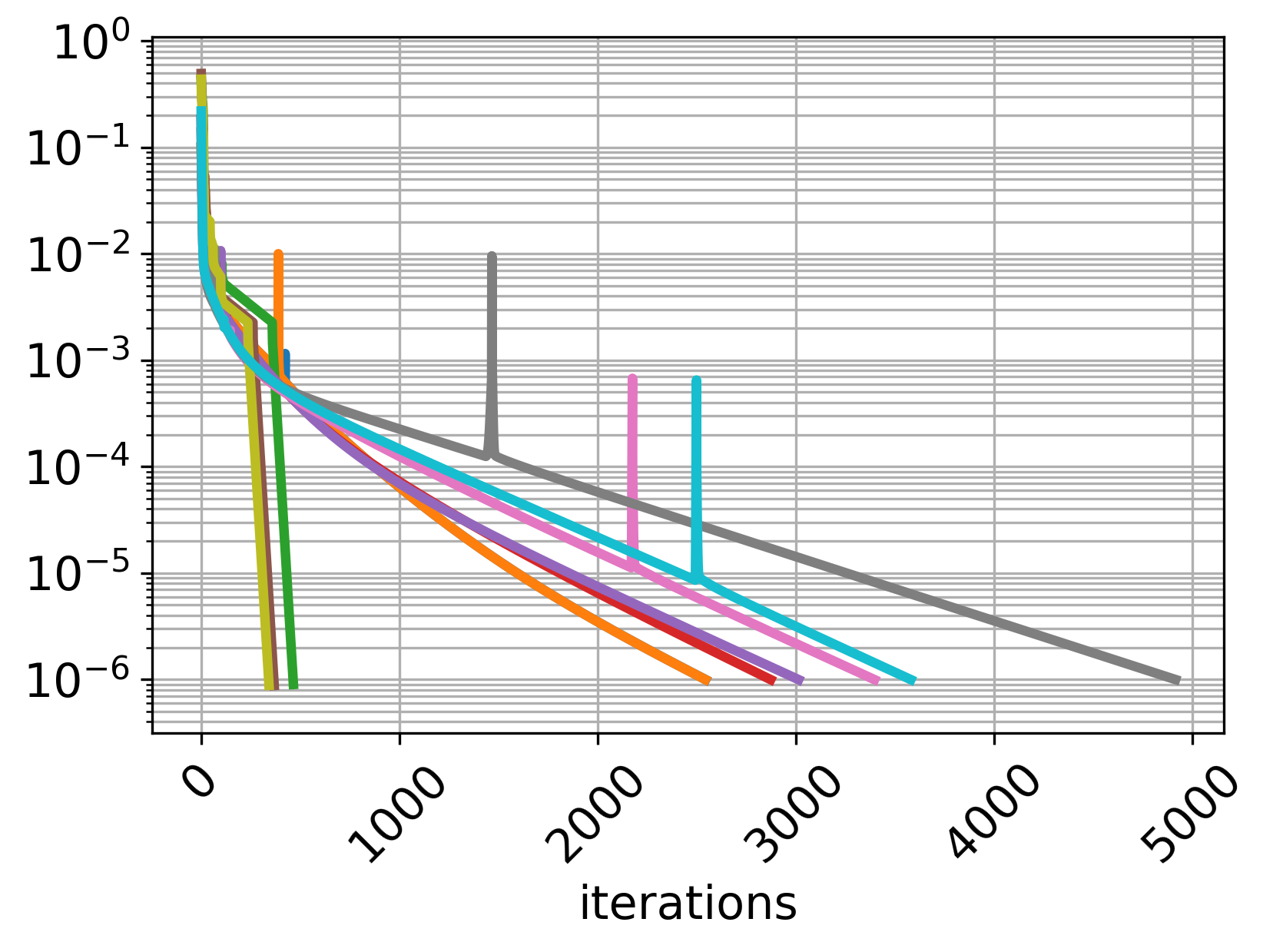}
		\caption{$\{\|x^{k+1} - x^k\|\}$}
		\label{fig:sub1b}
	\end{subfigure}
	\caption{Progression (log scale) along iterations of DR splitting using MCP penalty, $n=30$, $m=10$, $10$ different random starting points, $\lambda = 1$,  $\gamma = 0.9\cdot \frac{1}{2 L_f}$, $\theta = 1.5\gamma$, and $\sigma = 0.01$.} 
	\label{fig:rates-convergence-1}
\end{figure}

\begin{figure}[b!]
	\centering
	\begin{subfigure}{.5\textwidth}
		\centering
		\includegraphics[width=.95\linewidth]{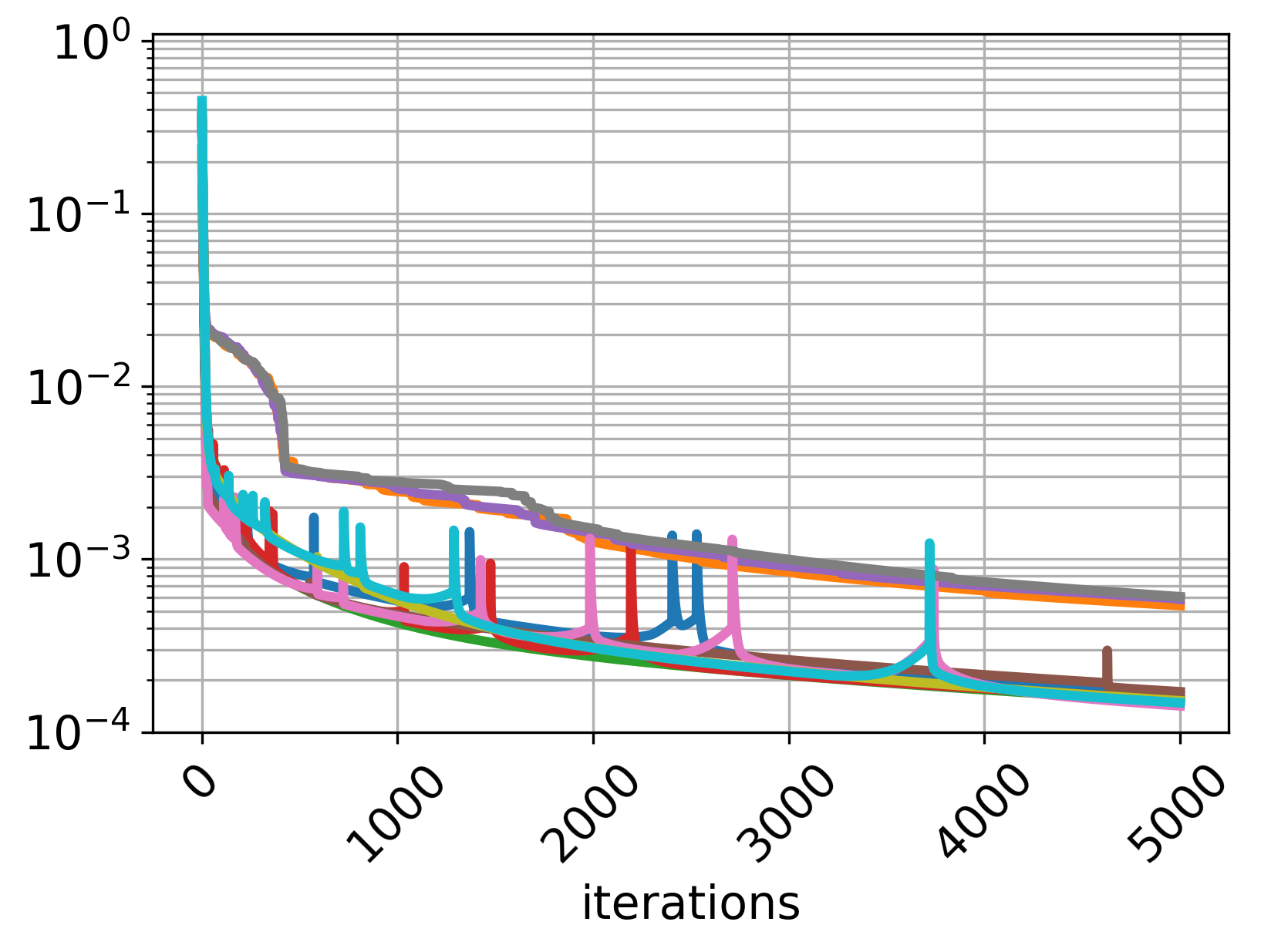}
		\caption{$\{\|z^{k+1}-z^k\|\}$.}
		\label{fig:sub2a}
	\end{subfigure}%
	\begin{subfigure}{.5\textwidth}
		\centering
		\includegraphics[width=.95\linewidth]{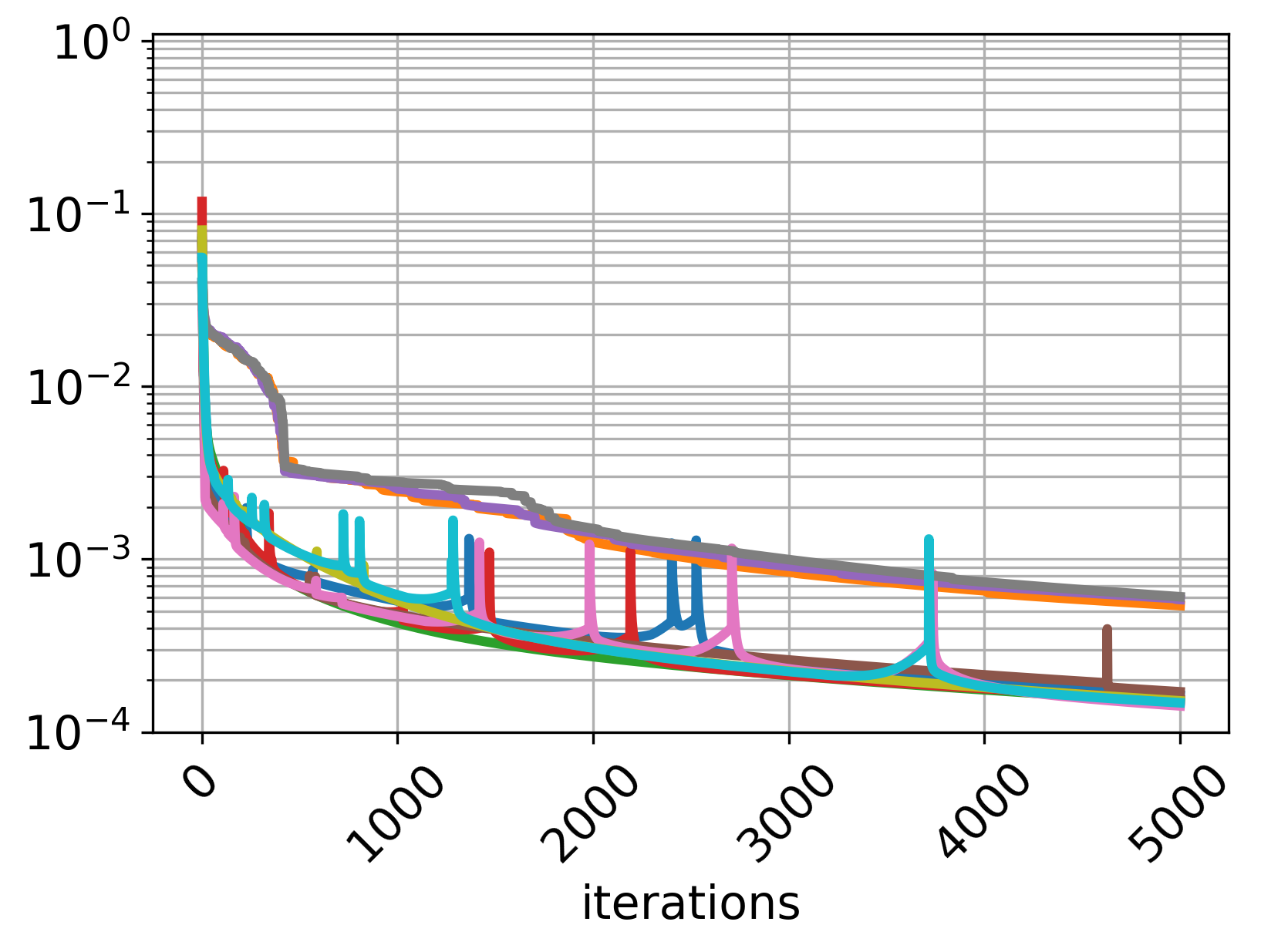}
		\caption{$\{\|y^{k+1} - y^k\|\}$}
		\label{fig:sub2b}
	\end{subfigure}
	\caption{Progression (log scale) along iterations of DR splitting using the SCAD penalty, $n=150$, $m=50$,  $10$ different random starting points, $\lambda = 0.5$,  $\gamma = 0.5\cdot \frac{1}{2 L_f}$, $\theta = 1.5\gamma$, and $\sigma = 0.1$.} 
	\label{fig:rates-convergence-2}
\end{figure}

\begin{figure}[h!]
	\centering
	\begin{subfigure}{.5\textwidth}
		\centering
		\includegraphics[width=.95\linewidth]{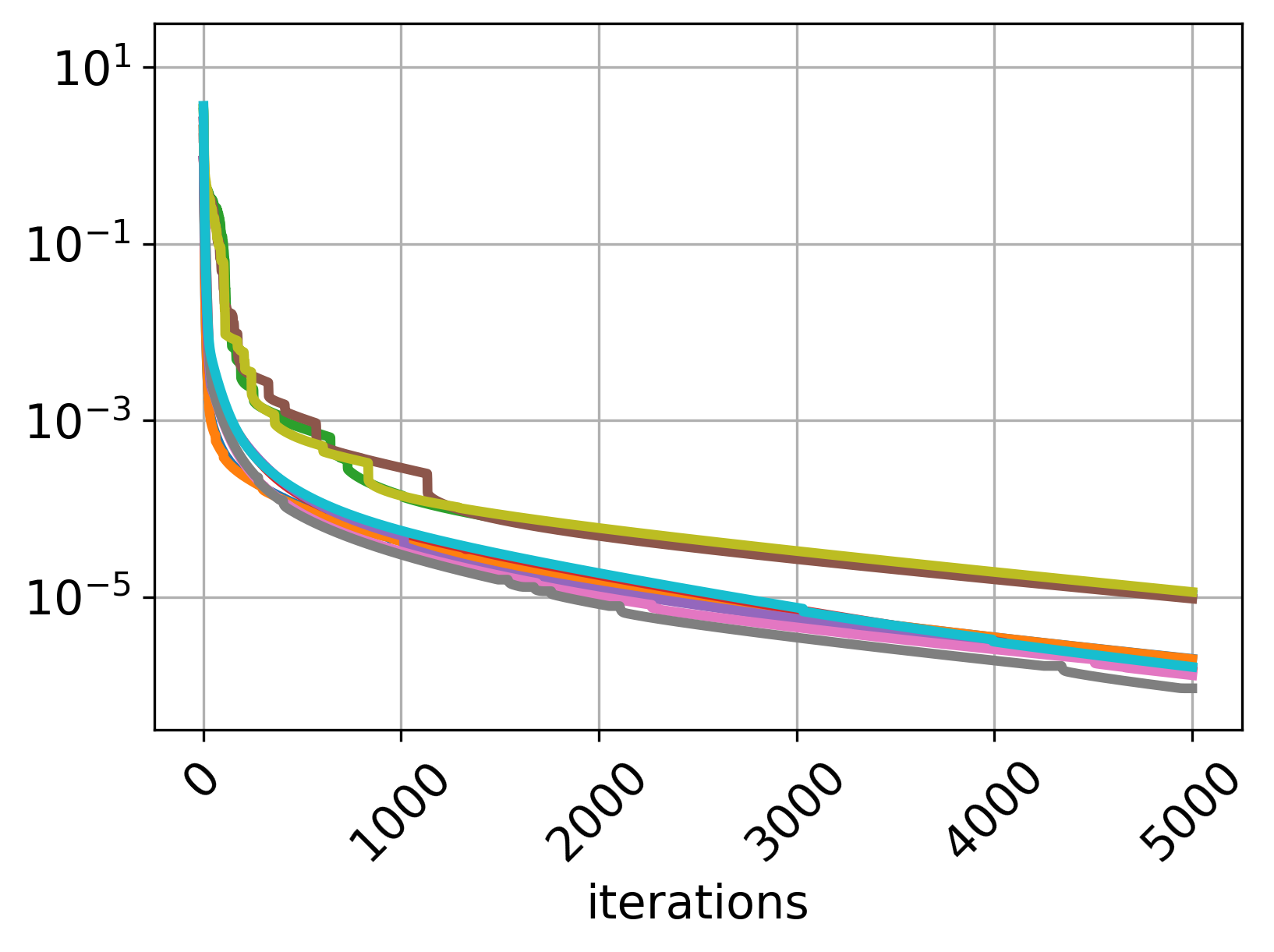}
		\caption{$\DRE(z^{k+1})-\DRE(z^k)$}
		\label{fig:sub3a}
	\end{subfigure}%
	\begin{subfigure}{.5\textwidth}
		\centering
		\includegraphics[width=.95\linewidth]{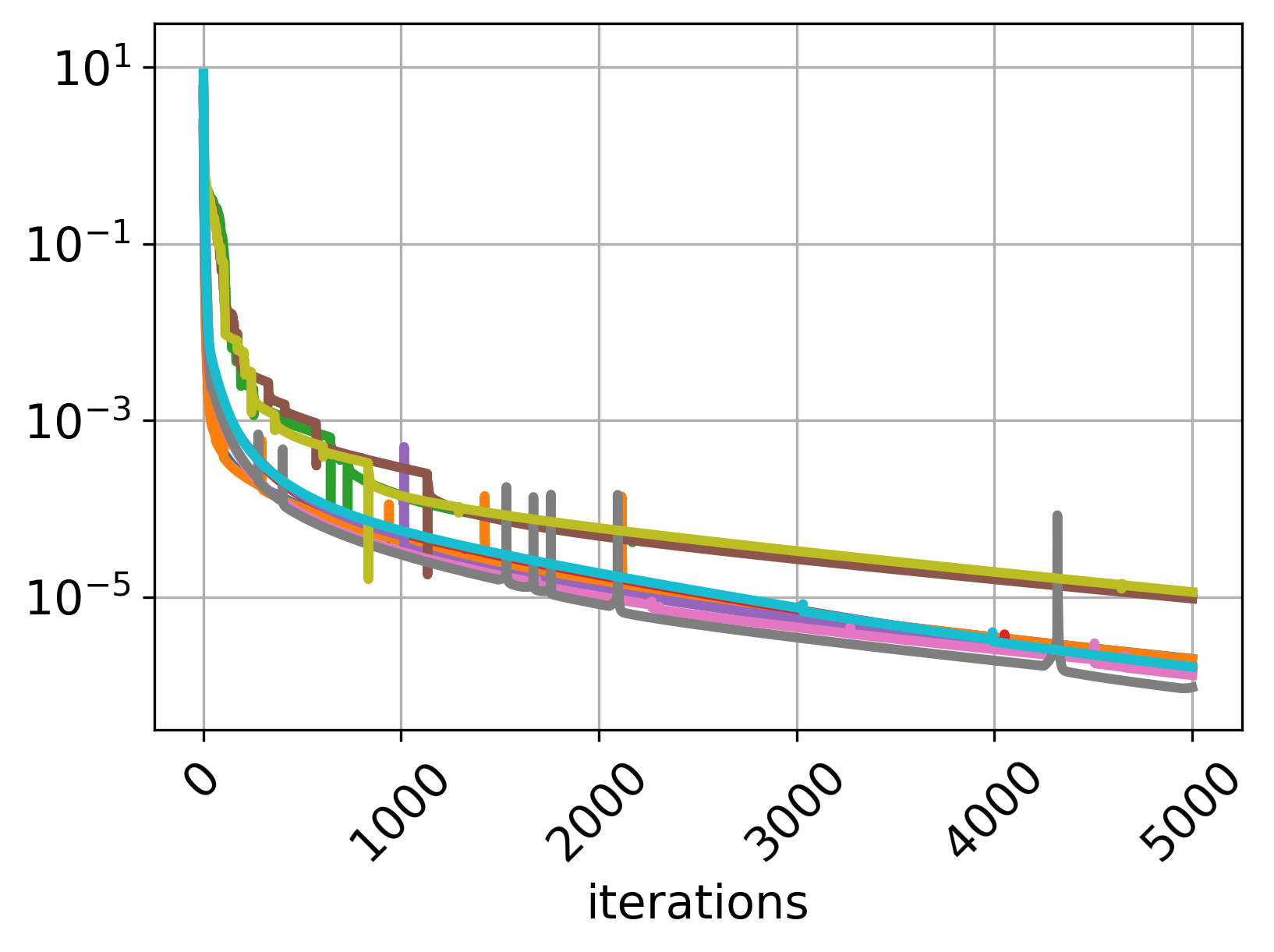}
		\caption{$|f(x^{k+1})+g(x^{k+1})-(f(x^k)+g(x^k))|$}
		\label{fig:sub3b}
	\end{subfigure}
	\caption{Progression (log scale) along iterations of DR splitting for different random starting points, MCP $n=90$, $m=30$, $\lambda = 0.5$,  $\gamma = 0.5\cdot \frac{1}{2 L_f}$, $\theta = 1.5\gamma$, $\sigma = 0.1$.} 
	\label{fig:rates-convergence-3}
\end{figure}

\begin{figure}[h!]
	\centering
	\begin{subfigure}{.5\textwidth}
		\centering
		\includegraphics[width=.95\linewidth]{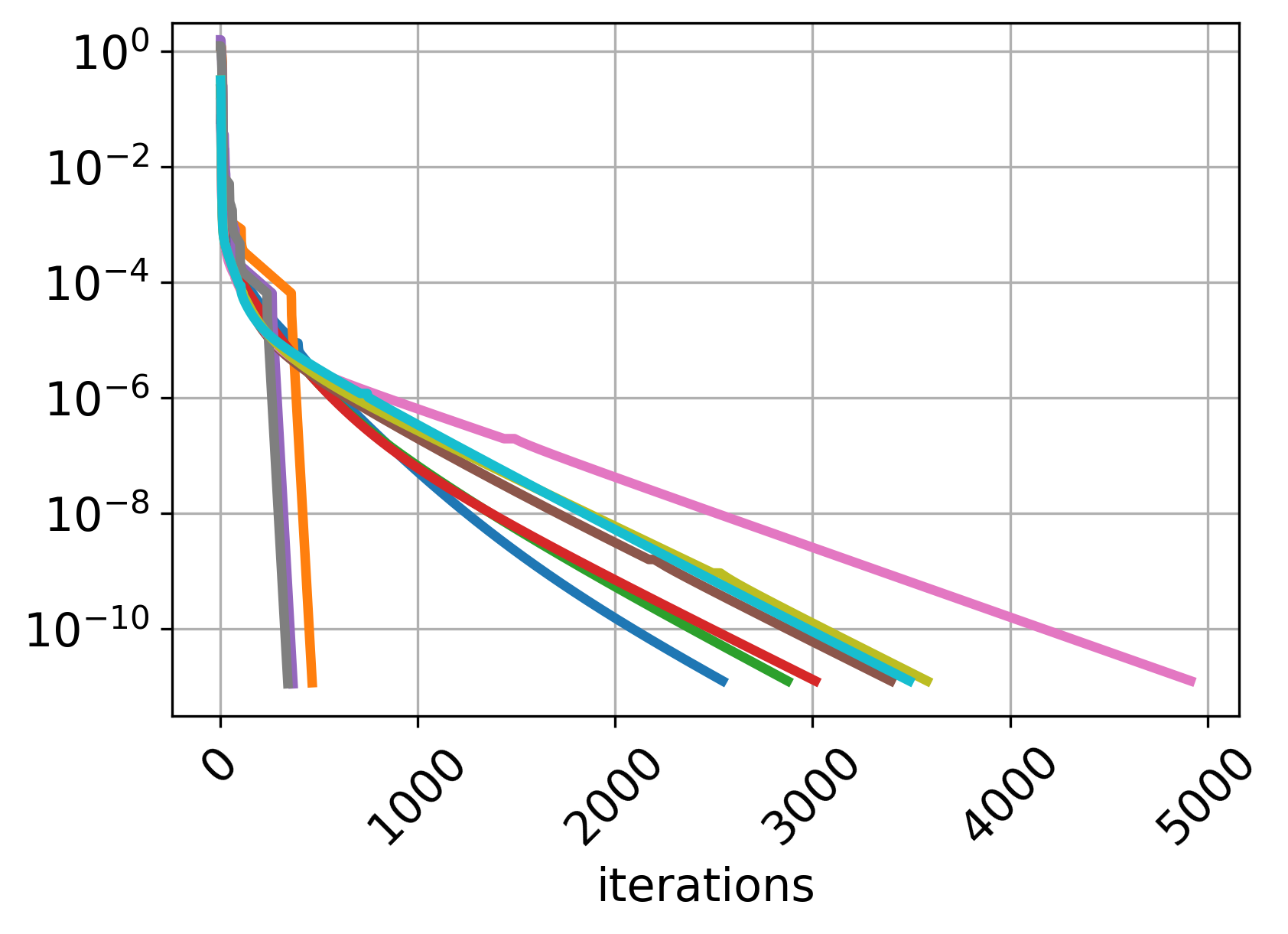}
		\caption{$\{\DRE(z^{k+1})-\DRE(z^k)\}$}
		\label{fig:sub4a}
	\end{subfigure}%
	\begin{subfigure}{.5\textwidth}
		\centering
		\includegraphics[width=.95\linewidth]{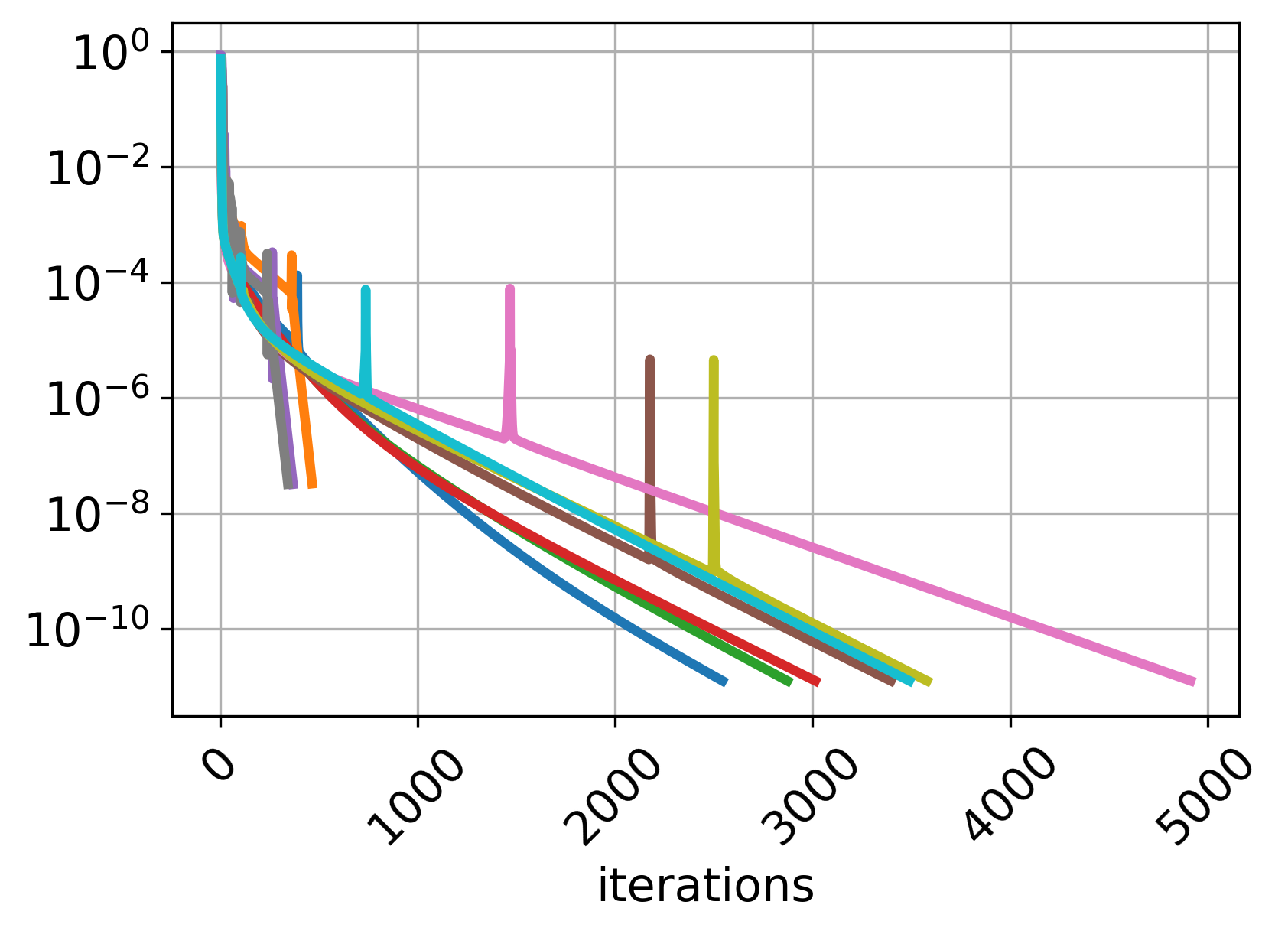}
		\caption{$\{|f(x^{k+1})+g(x^{k+1})-(f(x^k)+g(x^k))|\}$}
		\label{fig:sub4b}
	\end{subfigure}
	\caption{Progression (log scale) along iterations of DR splitting for $10$ different random starting points, SCAD $n=30$, $m=10$, $\lambda = 1$,  $\gamma = 0.9\cdot \frac{1}{2 L_f}$, $\theta = 1.5\gamma$, $\sigma = 0.01$.} 
	\label{fig:rates-convergence-4}
\end{figure}

Figures~\ref{fig:rates-convergence-1}--\ref{fig:rates-convergence-4} show the behavior of different measures of progress alongside iterations, for $10$ randomly generated initial points. Clearly, they show the anticipated linear rate of convergence of the method for the residual $\|x^k - y^k\|$, the distance of two consecutive iterates, and also the progress of the DRE and the value $\{f(x^k) + g(y^k)\}$ (cf. Theorem~\ref{DR:rate}).
 
 Finally, we report the performance profile for the stopping test mentioned above for MCP, SCAD and the $\ell_1$ norm as penalties. Just to illustrate, Table~\ref{tab:1} shows that for a particular configuration using the MCP penalty, in all the instances run the measure $\|x^k - y^k\|$ achieved the order of magnitude of $10^{-4}$, while only roughly $12\%$ of them surpassed the set threshold $10^{-6}$.

\begin{table}[h!]
	\centering
	\begin{tabular}{c||c|c|c|c|c|c|c}
		accuracy & $<10^{-1}$ &  $<10^{-2}$ & $<10^{-3}$  & $<10^{-4}$  & $<10^{-5}$  & $<10^{-6}$ & $<10^{-7}$  \\ \hline
		fraction & 1 & 1 & 1 & 1 &  0.511 & 0.237 & 0.126 \\ 
	\end{tabular}
	\caption{Fraction of problems solved by DR splitting using MCP regularizer for different levels of accuracy,  $\lambda \in \{0.5, 1.0, 1.5\}$,  $\gamma \in \{0.3, 0.5, 0.9\} \cdot \left( \frac{2-\lambda}{2 L_f} \right)$, $\sigma = 1$ and $10$ different initial points.}
	\label{tab:1}
\end{table}

Figure~\ref{fig:perf-prof-diff} is the performance profile comparing the two weakly convex regularizers and the convex $\ell_1$ regularizer in terms of the order of magnitude of $\{\|x^k - y^k\|\}$ by the final iteration performed. We observe that the SCAD penalty achieved an accuracy of $10^{-7}$ for slightly over $40\%$ of the instances, while MCP does it for $30\%$ to $40\%$ of the instances. Furthermore, both nonconvex penalties performed better than the $\ell_1$ norm for accuracy equal or better than $10^{-6}$.  For all instances, both MCP and SCAD showed accuracy of at least $10^{-4}$, while for the convex regularizer, around $10\%$ of the runs could not achieve this order of magnitude.  These results suggest that SCAD is more promising than the other two penalties, at the possible expense of having an arguably more intricate proximal operator to compute.

\begin{figure}[h!]
	\centering
	\includegraphics[width=0.5\textwidth]{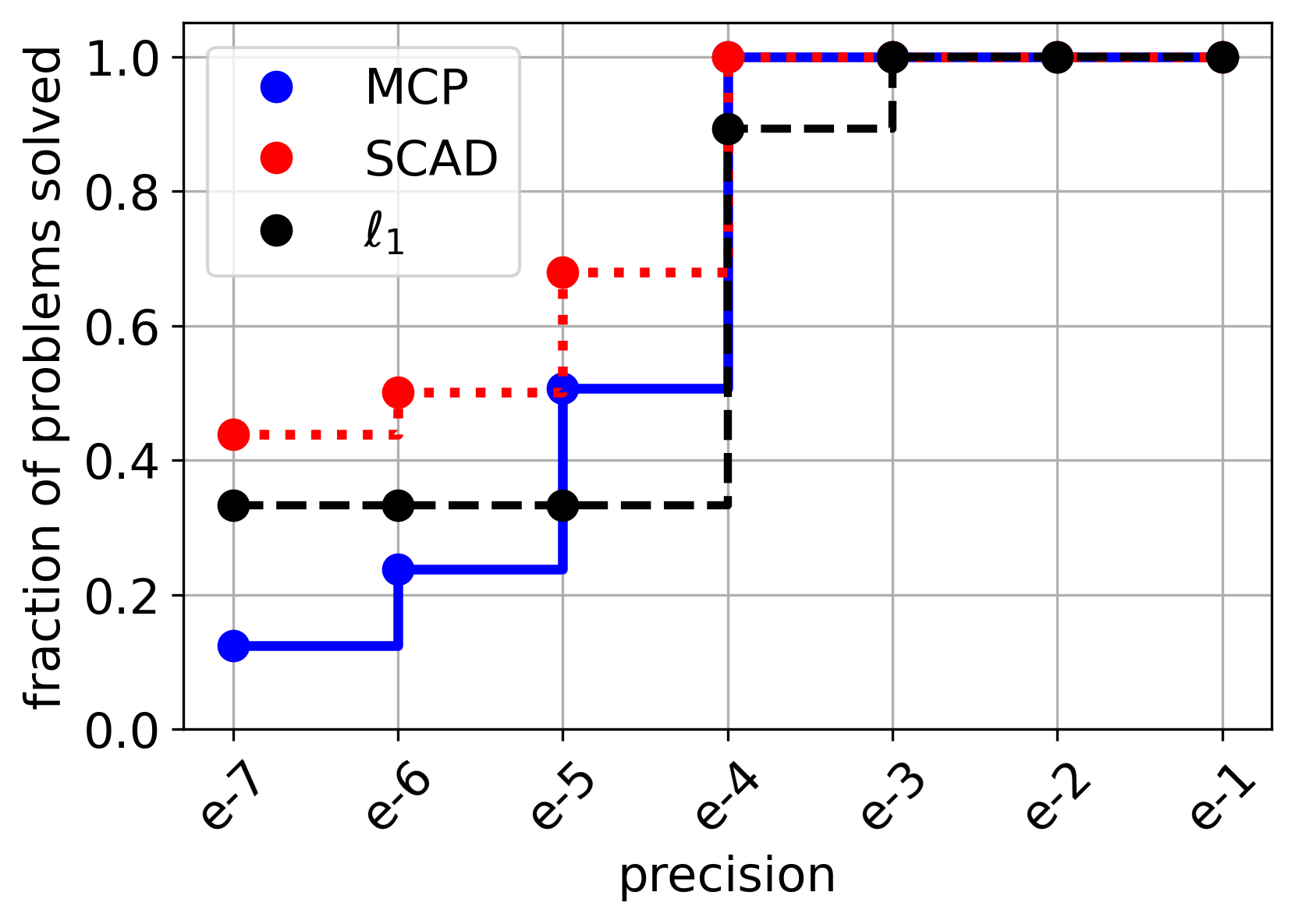}
	\caption{Performance profile of the Douglas-Rachford splitting method to solve the regularized least squares problem using convex and nonconvex regularizers.}
	\label{fig:perf-prof-diff}
\end{figure}


\section{Concluding remarks} \label{section-final}

In this article, we proved that the Moreau-type envelopes defined for the FB and the DR splitting methods approximate the objective function via epigraphs. This fact has a number of consequences, two of them discussed here. We then proceeded to analyze the behavior of the DR method itself via envelopes, showing that it is possible to retrieve linear convergence of the algorithm by resorting to techniques usually employed for descent methods. The author is currently working to extend these results to other types of splitting methods for weakly convex optimization.

\bibliographystyle{abbrv}
\bibliography{biblio}

\end{document}